\documentclass[12pt,oneside,reqno]{amsart}
\usepackage{}
\usepackage{amssymb}
\usepackage[colorlinks=true, linkcolor=blue, citecolor=red]{hyperref}
\usepackage{bbm}
\usepackage{cases}
\usepackage{amsmath}
\usepackage{graphicx}
\usepackage{mathrsfs}
\usepackage{stmaryrd}
\usepackage{color}
\usepackage{soul}
\usepackage[dvipsnames]{xcolor}
\usepackage{amsfonts}
\usepackage{enumerate,amsmath,amssymb,amsthm}

\numberwithin{equation}{section}

\newcommand{\be}{\begin{eqnarray}}
\newcommand{\mE}{\end{eqnarray}}
\newcommand{\ce}{\begin{eqnarray*}}
\newcommand{\de}{\end{eqnarray*}}
\newtheorem{theorem}{Theorem}[section]
\newtheorem{lemma}[theorem]{Lemma}
\newtheorem{remark}[theorem]{Remark}
\newtheorem{definition}[theorem]{Definition}
\newtheorem{proposition}[theorem]{Proposition}
\newtheorem{example}[theorem]{Example}
\newtheorem{corollary}[theorem]{Corollary}

\def\eps{\varepsilon}

\def\p{\partial}

\def\[{{\Big[}}
\def\]{{\Big]}}
\def\<{{\langle}}
\def\>{{\rangle}}
\def\({{\Big(}}
\def\){{\Big)}}

\def\bx{{\mathbf{x}}}

\def\dif{{\mathord{{\rm d}}}}

\def\no{\nonumber}
\def\={&\!\!=\!\!&}
\def\bt{\begin{theorem}}
\def\et{\end{theorem}}
\def\bl{\begin{lemma}}
\def\el{\end{lemma}}
\def\br{\begin{remark}}
\def\er{\end{remark}}

\def\bd{\begin{definition}}
\def\ed{\end{definition}}
\def\bp{\begin{proposition}}
\def\ep{\end{proposition}}
\def\bc{\begin{corollary}}
\def\ec{\end{corollary}}
\def\bx{\begin{example}}
\def\ex{\end{example}}

\def\cD{{\mathcal D}}

\def\cL{{\mathcal L}}

\def\mC{{\mathbb C}}

\def\mE{{\mathbb E}}

\def\mN{{\mathbb N}}

\def\mP{{\mathbb P}}

\def\mR{{\mathbb R}}

\def\sF{{\mathscr F}}

\def\sI{{\mathscr I}}
\def\sJ{{\mathscr J}}

\def\sL{{\mathscr L}}

\def\sN{{\mathscr N}}
\def\sO{{\mathscr O}}

\def\sU{{\mathscr U}}

\def\sX{{\mathscr X}}
\def\sY{{\mathscr Y}}
\def\sZ{{\mathscr Z}}

\def\geq{\geqslant}
\def\leq{\leqslant}

\allowdisplaybreaks

\begin{document}

\title{Asymptotic behavior of multiscale stochastic partial differential equations}

\date{}

\author{Michael R\"ockner,\,\, Longjie Xie\,\, and\,\, Li Yang}

\address{Michael R\"ockner:
	Fakult\"{a}t f\"{u}r Mathematik, Universit\"{a}t Bielefeld, D-33501 Bielefeld, Germany $\&$ Academy of Mathematics and Systems Science,
	Chinese Academy of Sciences (CAS), Beijing, 100190, P.R.China\\
	Email: roeckner@math.uni-bielefeld.de
}

\address{Longjie Xie:
	School of Mathematics and Statistics $\&$ Research Institute of Mathematical Science, Jiangsu Normal University,
	Xuzhou, Jiangsu 221000, P.R.China\\
	Email: longjiexie@jsnu.edu.cn
}

\address{Li Yang:
	School of Mathematics, Shandong University,
	Jinan, Shandong 250100, P.R.China\\
	Email: llyang@mail.sdu.edu.cn
}

\thanks{This work is supported  by the DFG through CRC 1283 and NSFC (No. 11701233, 11931004).}

\begin{abstract}
In this paper, we  study the asymptotic behavior of a semi-linear slow-fast stochastic partial differential
equation with singular coefficients. Using the Poisson equation in Hilbert space, we first establish the strong convergence in the averaging principe, which can be viewed as a functional law of large numbers. Then we study the stochastic fluctuations between the original system and its averaged equation. We show that the normalized difference  converges weakly to an Ornstein-Uhlenbeck type process, which can be viewed as a functional central limit theorem. Furthermore, rates of convergence both for the strong convergence and the normal deviation are obtained, and these convergence   are shown not to depend on the regularity of the coefficients in the equation for the fast variable, which coincides with the intuition, since in the limit systems the fast component has been totally averaged or homogenized
out.
	\bigskip

	
	\vspace{2mm}
\end{abstract}

\maketitle

\tableofcontents

\section{Introduction}

Consider the following fully coupled slow-fast stochastic partial differential equation (SPDE for short) in $H_1\times H_2$:
\begin{equation} \label{spde1}
\left\{ \begin{aligned}
&\dif X^{\eps}_t =AX^{\eps}_t\dif t+F(X^{\eps}_t, Y^{\eps}_t)\dif t+\dif W^1_t,\quad\qquad\qquad\quad\!\! X^{\eps}_0=x\in H_1,\\
&\dif Y^{\eps}_t =\eps^{-1}BY^{\eps}_t\dif t+\eps^{-1}G(X^{\eps}_t, Y^{\eps}_t)\dif t+\eps^{-1/2} \dif W_t^2,\quad Y^{\eps}_0=y\in H_2,
\end{aligned} \right.
\end{equation}
where $H_1, H_2$ are two Hilbert spaces, $A: \cD(A)\subset H_1\to H_1$ and $B: \cD(B)\subset H_2\to H_2$ are linear operators, $F: H_1\times H_2\to H_1$ and $G: H_1\times H_2\to H_2$ are reaction coefficients, $W^1_t$ and $W^2_t$ are mutually independent $H_1$- and $H_2$-valued $(\sF_t)$-Wiener processes both defined on some probability space $(\Omega,\mathscr{F},\mP)$ with a normal filtration $(\sF_t)_{t\geq0}$, and the small parameter $0<\eps\ll 1$ represents the separation  of  time scales  between   the slow process $X_t^\eps$ (which can be thought
of as the mathematical model for a phenomenon appearing at the natural time scale) and  the fast motion $Y_t^\eps$ (with time order $1/\eps$, which can be interpreted as the fast environment). Such multi-scale models appear frequently in many real-world  dynamical systems.  Typical examples   include  climate weather interactions  (see e.g. \cite{Ki,MTV}), macro-molecules (see e.g. \cite{BKRP,KK}), geophysical fluid flows (see e.g. \cite{GD}), stochastic volatility in finance (see e.g. \cite{FFK}), etc.
However, it is often too difficult to analyze or   simulate the underlying system (\ref{spde1}) directly due to  the two widely separated time scales and the cross interactions between the slow  and fast  modes. Thus a simplified equation
which governs the evolution of the system over a long time scale is highly desirable and is quite important for applications.

\vspace{1mm}
It is known that under suitable regularity assumptions on the coefficients, the slow process $X_t^\eps$ will converge strongly (in the $L^2(\Omega)$-sense) to the solution of the following reduced equation:
\begin{align}\label{spde2}
\dif \bar{X}_t=A\bar{X}_t\dif t+\bar{F}(\bar{X}_t)\dif t+\dif W_t^1,\quad \bar X_0=x\in H_1,
\end{align}
where the averaged coefficient is given by
\begin{align}\label{df1}
\bar{F}(x):=\int_{H_2}F(x,y)\mu^x(\dif y),
\end{align}
and $\mu^x(\dif y)$ is the unique invariant measure  of the process
 $Y_t^x$, which is the solution of the frozen equation
\begin{align}\label{froz}
\dif Y_t^x=BY_t^x \dif t+G(x,Y_t^x)\dif t+\dif W_t^2, \quad Y_0^x=y\in H_2.
\end{align}
The effective system (\ref{spde2}) then captures the  essential dynamics of the system (\ref{spde1}), which does not depend on the fast variable any more and thus is much simpler than SPDE (\ref{spde1}). This theory, known as the averaging principle, was first developed for deterministic systems by Bogoliubov \cite{BM}, and extended to stochastic differential equations (SDEs for short)  by Khasminskii \cite{K1}. In the past decades, the averaging principle for systems with a finite number of degrees of freedom has been intensively studied, see e.g. \cite{BK,GR,HLi,KY,V0} and the references therein.
Passing from the finite dimensional to the infinite dimensional setting is more difficult, and the existing results  in the literature are relatively few. In \cite{CF}, Cerrai and Freidlin proved the averaging principle for slow-fast stochastic reaction-diffusion system  where there is no noise in the slow equation. Later, Cerrai \cite{Ce,C2} generalized this result to general reaction-diffusion equations with multiplicative noise and  coefficients of polynomial growth , see also \cite{CL,DS,FWL0} for further developments.
We also mention that in these results, no rates of convergence in terms of  $\eps\to0$ are provided. But
for numerical purposes, it is important to know the rate of convergence of the slow variable to the effective dynamic. The main motivation comes from the well-known Heterogeneous Multi-scale Methods  used to approximate the slow component in system (\ref{spde1}), see e.g. \cite{EL,KT}.
In this direction,  Br\'ehier \cite{Br1} first studied the rates of strong convergence for the averaging principle of SPDEs with  noise only in the fast motion, and  $(\frac{1}{2}$-)-order of convergence is obtained. Extensions to  general stochastic reaction-diffusion equations are made in \cite{WR}, and  $\frac{1}{2}$-order of convergence is obtained. For more recent results, we refer the interested readers to the work \cite{Br2}  and the references therein.

\vspace{1mm}
The strong convergence in the averaging principle can be viewed as a functional law of large numbers. Once we  obtain the validity of the averaging principle, it is natural to go one step further to consider the functional central limit theorem. Namely, to study the small fluctuations of the original system (\ref{spde1}) around its averaged equation (\ref{spde2}).
To leading order, these fluctuations can be captured by characterizing the asymptotic behavior of the normalized difference
\begin{align}\label{zt}
Z_t^{\eps}:=\frac{X_t^\eps-\bar X_t}{\sqrt{\eps}}
\end{align}
as $\eps$ tends to 0. Under extra regularity assumptions on the coefficients, the deviation process $Z_t^\eps$ is known to converge weakly (in the distribution sense) towards a Gaussian  process $\bar Z_t$, whose covariance can be described explicitly.  Such result, also known as the Gaussian approximation, is closely related to the homogenization  for solutions of  partial differential equations with singularly perturbed terms, which has its own interest in the theory of PDEs, see e.g. \cite{HP,HP2} and \cite[Chapter IV]{FW}.    For the study of normal deviations of multi-scale SDEs, we refer the readers to the fundamental paper by Khasminskii \cite{K1},  see also \cite{P-V,P-V2,RX}  for further developments. In the infinite dimensional situation, as far as we know,  there exist only two papers. Cerrai \cite{Ce2} studied the normal deviations for slow-fast SPDEs in a special case, i.e., a deterministic reaction-diffusion equation with one dimensional space variable perturbed by a fast motion. Later, this was generalized to general  stochastic reaction-diffusion equations by Wang and Roberts \cite{WR}. In both papers the methods of proof are based on the time discretisation procedure  which involve some complicated  tightness arguments. We point out  that besides having intrinsic interest, the functional central limit theorem  is also useful in applications. In particular, we can get the formal asymptotic expansion
$$
X_t^\eps\stackrel{\cD}{\approx} \bar X_t+\sqrt{\eps}\bar Z_t,
$$
where $\stackrel{\cD}{\approx}$ means approximate equality of probability distributions. Such  expansion has been introduced in
the context of stochastic climate models. In physics
this is also called the Van Kampen's  scheme (see e.g. \cite{Ar,HK}),
which provides  better approximations for the original system (\ref{spde1}).

\vspace{1mm}
In the present paper, we shall first establish a stronger convergence result in the averaging principle  for SPDE (\ref{spde1}). More precisely, we show that for any $T>0,q\geq 1$ and $\gamma\in[0,1/2)$,  there exists a constant $C_T>0$ such that
	$$
	\sup\limits_{t\in[0,T]}\mE\|(-A)^\gamma( X^{\eps}_t-\bar X_t)\|^q\leq C_{T}\,\eps^{\frac{q}{2}},
	$$
see {\bf Theorem \ref{main1}} below. Compared with the existing results in the literature, we assume that the coefficients are only H\"older continuous with respect to the fast variable,  and we obtain not only the strong convergence in $L^q(\Omega)$-sense with any $q\geq1$, but also in $\|\cdot\|_{(-A)^\gamma}$ norm with any $\gamma\in [0,1/2)$, which is particularly interesting  for SPDEs in comparison with the finite dimensional setting  since $A$ is an unbounded operator and seems to have never been obtained before. Moreover, the $\frac{1}{2}$-order rate of convergence is also obtained, which is known to be optimal (when $\gamma=0$). In particular, we show that the convergence in the averaging principle does not depend on the regularity of the coefficients with respect to the fast variable. This coincides  with the intuition, since in the limit equation the fast component has been totally averaged out. We point out that the strong convergence of $(-A)^\gamma X_t^\eps$ to $(-A)^\gamma\bar X_t$ will play an important role below in our study of the homogenization for the normalized difference  $Z_t^\eps$. Furthermore,  the index $\gamma<1/2$ should be the best  possible, see  Remark \ref{r1} for more detailed explanations.

\vspace{1mm}
The argument we shall use to establish the above strong convergence is different from those in \cite {Br1,Ce,C2,CF,CL,DS,FWL0,WR}, where the classical Khasminskii's time discretisation procedure is used.
Our method is based on the Poisson equation. More precisely, consider the following Poisson equation in the Hilbert space $H_1\times H_2$:
\begin{align}\label{intp}
\cL_2(x,y)\psi(x,y)=-\phi(x,y),\quad y\in H_2,
\end{align}
where $\cL_2(x,y)$ is an ergodic elliptic operator with respect to the $y$ variable (see (\ref{L2}) below), $x\in H_1$ is regarded as a parameter, and $\phi: H_1\times H_2\to \mR$ is a measurable function. Such kind of equation, i.e., with a parameter and in the whole space (without boundary condition), has been studied only  relatively recently and is now realized  to be  very important in the theory of limit theorems in probability theory and numerical
approximation for time-averaging estimators and invariant measures, see e.g. \cite{MST,PP}. In the finite dimensional situation, equations of the form (\ref{intp}) have been studied in a series of papers by  Pardoux and Veretennikov \cite{P-V,P-V2,P-V3}, see also \cite{RX} and the references therein for further developments. Undoubtedly, extension to the infinite dimensional setting will be more difficult due to the unboundedness of the involved  operators.
In the recent work \cite{Br2}, the author studies  the rate of convergence in the averaging principle for slow-fast SPDEs with regular coefficients by assuming the solvability of the corresponding Poisson equation as well as regularity properties of the solutions. In addition, the SPDE considered therein is not fully coupled, i.e., the fast component $Y_t^\eps$ does not depend  on the slow process $X_t^\eps$, and  the two Hilbert spaces $H_1$, $H_2$ and the unbounded operators $A$, $B$ are assumed to be the same, which are used in the whole proof in an essential way.
Here, we shall establish the well-posedness of the Poisson equation (\ref{intp}) with only H\"older coefficients and in general Hilbert spaces $H_1\times H_2$, and study the regularity properties of the unique solution with respect to both the $y$-variable and the parameter $x$, see {\bf Theorem \ref{PP}} below, which should be  of independent interest.
Then, we  use the Poisson equation to derive a strong fluctuation estimate (see Lemma \ref{strf}) for an integral functional of the slow-fast   SPDE (\ref{spde1}). The strong convergence in the averaging principle with optimal rate of convergence  then follows directly. In addition, we also  provide a simple way to verify the regularity of the averaged coefficients by using Theorem \ref{PP} (see Lemma \ref{aveF} below), which is a  separate problem that one always encounters in the study of averaging principles, central limit theorems, homogenization and other limit theorems.

\vspace{1mm}

Next, we proceed to study the small fluctuations of the slow process $X_t^\eps$ around its average $\bar X_t$, i.e., we are interested in the homogenization behavior for $Z_t^\eps$ which is defined by (\ref{zt}). In view of (\ref{spde1}) and (\ref{spde2}), we have
\begin{align}\label{zte}
\dif Z^\eps_t&=A Z^\eps_t\dif t+\frac{1}{\sqrt{\eps}}\Big[F(X_t^\eps,Y_t^\eps)-\bar F(\bar X_t)\Big]\dif t\no\\
&=A Z^\eps_t\dif t+\frac{1}{\sqrt{\eps}}\Big[\bar F(X^\eps_t)-\bar F(\bar X_t)\Big]\dif t+\frac{1}{\sqrt{\eps}}\delta F(X_t^\eps,Y_t^\eps)\dif t,
\end{align}
where
\begin{align}\label{dF}
\delta F(x,y):=F(x,y)-\bar F(x).
\end{align}
We demonstrate that $Z_t^\eps$ converges weakly  to an  Ornstein-Uhlenbeck type process $\bar Z_t$ which satisfies the following linear SPDE:
\begin{align*}
\dif \bar Z_t=A\bar Z_t\dif t+D_x\bar F(\bar X_t).\bar Z_t\dif t+\sigma(\bar X_t)\dif \tilde W_t,
\end{align*}
where   $\tilde W_t$ is another cylindrical  Wiener process which is independent of $W_t^1$, and the diffusion coefficient $\sigma$ is Hilbert-Schmidt operator valued and  given by (\ref{sst}), see {\bf Theorem \ref{main3}} below. Compared with \cite{Ce2,WR}, our system (\ref{spde1}) is more general, and the coefficients are assumed to be only H\"older continuous with respect to the fast variable, and we provide a more precise formula for the new diffusion coefficient $\sigma$.  Moreover, the arguments we use to prove the above convergence is different from \cite{Ce2,WR}, and in addition the rate of convergence is obtained, which does not depends on the regularity of the coefficients with respect to the fast variable.

\vspace{1mm}
It turns out that our method to prove the above functional central limit theorem is closely and universally connected with the proof of the strong convergence in the averaging principle.  Namely,  we shall first use  the result on the Poisson equation (\ref{intp}) established in Theorem \ref{PP}  to derive some weak fluctuation estimates (see Lemma \ref{weaf}) for an integral functional involving the processes  $(X_t^\eps, Y_t^\eps)$ and $Z_t^\eps$. Combining  with the Kolmogorov equation associated with the process $(\bar X_t, \bar Z_t)$, we  prove the weak convergence of $Z_t^\eps$ to $\bar Z_t$  directly,  and rate of convergence is obtained as easy by-product. In addition, it will be quite easy to capture the structure of the homogenization limit $\bar Z_t$ from our arguments.
Here, we note that the whole system of equations satisfied by $(\bar X_t, \bar Z_t)$ is an SPDE with multiplicative noise. Even though infinite dimensional Kolmogorov equations with nonlinear diffusion coefficients have been studied very recently in \cite{Br4}, the regularity of the solutions obtained therein are not sufficient for our purpose. Thus, we derive some new regularity for the solution with respect to the $z$ variable (see Theorem \ref{lako} below), and develop a  trick in the proof of Theorem \ref{main3} to avoid using the regularity for the solution with respect to the $x$ variable.  Our approach can also be  adapted to study the normal deviations  for other classes of multi-scale SPDEs. We shall address these problems in future works.

\vspace{1mm}
The rest of this paper is organized as follows. In Section 2, we introduce some assumptions and state our main results. Section 3 is devoted to study the  Poisson equation in Hilbert spaces. Then, we prove the strong convergence result, Theorem \ref{main1}, and the normal deviation result, Theorem \ref{main3},  in Section 4 and Section 5, respectively. Finally, in the Appendix we prove some necessary estimates for the solution of the multiscale system (\ref{spde1}), which are slight generalizations of the existing results in the literature.
Throughout this paper, the letter $C$ with or without subscripts
will denote a positive constant, whose value may change in different places, and whose
dependence on parameters can be traced from the calculations.

\vspace{1mm}
{\bf Notations:}
To end this section, we introduce some notations, which will be used throughout this paper. Let $H_1, H_2$ and $H$ be three Hilbert spaces endowed with the scalar products $\langle\cdot,\cdot\rangle_1$, $\langle\cdot,\cdot\rangle_2$ and $\langle\cdot,\cdot\rangle_H$, respectively. The corresponding norms will be denoted by $\|\cdot\|_1,\|\cdot\|_2$ and $\|\cdot\|_H$. We use $\sL(H_1,H_2)$ to denote the space of all linear and bounded operators from $H_1$ to $H_2$. If $H_1=H_2,$ we write $\sL(H_1)=\sL(H_1,H_1)$  for simplicity. Recall that an operator $Q\in\sL(H)$ is called Hilbert-Schmidt if
$$
\Vert Q\Vert_{\mathscr{L}_2(H)}^2:=Tr(QQ^{*})<+\infty.
$$
We shall denote  the space of all Hilbert-Schmidt operators on $H$ by $\mathscr{L}_2(H)$.

\vspace{1mm}
For any $x\in H_1$, $y\in H_2$ and $\phi: H_1\times H_2\to H$,  we say that $\phi$ is G\^ateaux differentiable at $x$ if there exists a $D_x\phi(x,y)\in \sL(H_1,H)$ such that for all $h_1\in H_1$,
$$
\lim_{\tau\to 0}\frac{\phi(x+\tau h_1,y)-\phi(x,y)}{\tau}=D_x\phi(x,y).h_1.
$$
If in addition
$$
\lim_{\|h_1\|_1\to 0}\frac{\|\phi(x+h_1,y)-\phi(x,y)-D_x\phi(x,y).h_1\|_H}{\|h_1\|_1}=0,
$$
$\phi$ is called Fr\'echet differentiable at $x$. Similarly, for any $k\geq 2$ we can define the $k$ times G\^ateaux and Fr\'echet derivative of $\phi$ at $x$, and we will identify the higher order  derivative $D^k_x\phi(x,y)$ with a linear operator in $\sL^k(H_1,H):=\sL(H_1,\sL^{(k-1)}(H_1,H))$, endowed with the operator norm
\begin{align*}
\|D_x^k\phi(x,y)\|_{\sL^k(H_1,H)}:=\sup_{\|h_1\|_1,\|h_2\|_1,\cdots,\|h_k\|_1, \|h\|_H\leq 1}\!\<D^k_x\phi(x,y).(h_1,h_2,\cdots,h_k),h\>_H.
\end{align*}
 In the same way, we define the G\^ateaux and Fr\'echet derivatives of $\phi$ with respect to the $y$ variable, and we have $D_y\phi(x,y)\in \sL(H_2,H)$, and for $k\geq 2$, $D^k_y\phi(x,y)\in \sL^k(H_2,H) :=\sL(H_2,\sL^{(k-1)}(H_2,H))$.

\vspace{1mm}
We will denote by $L^\infty_p(H_1\times H_2,H)$  the space of  all measurable maps $\phi: H_1\times H_2\to H$ with linear growth in $x$ and polynomial growth in $y$, i.e., there exists a constant $p\geq 1$  such that
$$
\|\phi\|_{L^\infty_p(H)}:=\sup_{(x,y)\in H_1\times H_2}\frac{\|\phi(x,y)\|_H}{1+\|x\|_1+\|y\|_2^p}<\infty.
$$
For $k\in\mN$, the space $C^{k,0}_p(H_1\times H_2,H)$ contains all maps $\phi\in L^\infty_p(H_1\times H_2,H)$ which are $k$ times G\^ateaux differentiable at any $x\in H_1$ and
$$
\|\phi\|_{C_p^{k,0}(H)}:=\sup_{(x,y)\in H_1\times H_2}\frac{\sum\limits_{\ell=1}^k\|D_x^\ell\phi(x,y)\|_{\sL^\ell(H_1,H)}}{1+\|y\|_2^p}<\infty.
$$
Similarly, the space $C^{0,k}_p(H_1\times H_2,H)$ consists of all maps $\phi\in L^\infty_p(H_1\times H_2,H)$  which are $k$ times G\^ateaux differentiable at any $y\in H_2$ and
\begin{align}\label{norm}
\|\phi\|_{C_p^{0,k}(H)}:=\sup_{(x,y)\in H_1\times H_2}\frac{\sum\limits_{\ell=1}^k\|D_y^\ell\phi(x,y)\|_{\sL^\ell(H_2,H)}}{1+\|x\|_1+\|y\|_2^p}<\infty.
\end{align}
 We also introduce the space $\mC^{0,k}_p(H_1\times H_2,H)$ consisting of all maps which are $k$ times Fr\'echet differentiable at any $y\in H_2$ and satisfies (\ref{norm}).
For $k,\ell\in\mN$, let  $C^{k,\ell}_p(H_1\times H_2,H)$ be the space of all maps satisfying
$$
\|\phi\|_{C_p^{k,\ell}(H)}:=\|\phi\|_{L^\infty_p(H)}+\|\phi\|_{C_p^{k,0}(H)}+\|\phi\|_{C_p^{0,\ell}(H)}<\infty,
$$
and for $\eta\in(0,1)$, we use $C_p^{k,\eta}(H_1\times H_2,H)$ to denote the subspace of $C_p^{k,0}(H_1\times H_2,H)$ consisting of all maps such that
$$
\|\phi(x,y_1)-\phi(x,y_2)\|_H\leq C_0\|y_1-y_2\|_2^\eta\big(1+\|x\|_1+\|y_1\|_2^p+\|y_2\|_2^p\big).
$$
When the subscript $p$ is replaced by $b$ in the notations for above spaces, we mean that the map itself and its derivatives are all bounded. When $H=\mR$, we will omit the letter $H$ in the above notations  for simplicity.

\vspace{1mm}

\section{Statement of the main results}

\subsection{Assumptions and preliminaries}

For $i=1,2$, let $\{e_{i,n}\}_{n\in \mN}$ be a complete orthonormal basis of $H_i$. We assume that the two unbounded linear operators $A$ and $B$, with domains  $\cD(A)$ and $\cD(B)$, satisfy the following condition:

 \vspace{2mm}
\noindent{\bf (A1):} There exist non-decreasing sequences of real positive numbers $\{\alpha_n\}_{n\in\mN}$ and $\{\beta_n\}_{n\in\mN}$ such that
\begin{align}\label{bnbn}
Ae_{1,n}=-\alpha_ne_{1,n},\quad Be_{2,n}=-\beta_ne_{2,n},\quad \forall n\in\mN.
\end{align}

 \vspace{1mm}
\noindent In this setting, the powers of  $-A$ and $-B$ can be easily defined as follows: for any $\theta\in[0,1]$,
$$
(-A)^\theta x:=\sum_{n\in\mN}\alpha_n^{\theta}\langle x,e_{1,n}\rangle_1 e_{1,n}\quad\text{and}\quad (-B)^\theta y:=\sum_{n\in\mN}\beta_n^{\theta}\langle y,e_{2,n}\rangle_2 e_{2,n},
$$
with domains
$$
\cD((-A)^\theta):=\bigg\{x\in H_1: \|x\|_{(-A)^\theta}^2:=\sum\limits_{n\in\mN}\alpha_n^{2\theta}\langle x,e_{1,n}\rangle_1^2<\infty\bigg\}
$$
and
$$
\cD((-B)^\theta):=\bigg\{y\in H_2: \|y\|_{(-B)^\theta}^2:=\sum\limits_{n\in\mN}\beta_n^{2\theta}\langle y,e_{2,n}\rangle_2^2<\infty\bigg\}.
$$
Moreover, the corresponding semigroups $\{e^{tA}\}_{t\geq0}$ and $\{e^{tB}\}_{t\geq0}$  can be defined through the following spectral formulas: for any $t\geq 0$, $x\in H_1$  and $y\in H_2$,
$$
e^{tA}x:=\sum_{n\in\mN}e^{-\alpha_nt}\langle x,e_{1,n}\rangle_1\,e_{1,n}\quad\text{and}\quad e^{tB}y:=\sum_{n\in\mN}e^{-\beta_nt}\langle y,e_{2,n}\rangle_2\,e_{2,n}.
$$
The following regularization properties for these semigroups are more or less standard.
We write them for $e^{tA}$, but they also hold for $e^{tB}$.

\bp\label{arp}
Let $\gamma\in[0,1]$ and $\theta\in[0,\gamma]$. We have:

\vspace{2mm}
\noindent
(i) For any $t>0$ and $x\in \cD((-A)^\theta)$,
\begin{align}\label{pp1}
\|e^{tA}x\|_{(-A)^\gamma}\leq C_{\gamma,\theta}t^{-\gamma+\theta}e^{-\frac{\alpha_1}{2}t}\|x\|_{(-A)^\theta};
\end{align}
(ii) For any $0\leq s\leq t$  and $x\in H_1$,
\begin{align}\label{pp2}
\|e^{tA}x-e^{sA}x\|_1\leq C_\gamma(t-s)^{\gamma}\|e^{sA}x\|_{(-A)^\gamma};
\end{align}
(iii) For any $0<s\leq t$ and $x\in \cD((-A)^\theta)$,
\begin{align}\label{pp3}
\|e^{tA}x-e^{sA}x\|_1\leq C_{\gamma,\theta}\frac{(t-s)^{\gamma}}{s^{\gamma-\theta}}e^{-\frac{\alpha_1}{2}s}\|x\|_{(-A)^\theta},
\end{align}
where $\alpha_1$ is the smallest eigenvalue of $A$, and $C_\gamma, C_{\gamma,\theta}>0$ are constants.
\ep

\begin{proof}
For any $t>0$ and $x\in \cD((-A)^\theta),$ we have
\begin{align*}
&\|(-A)^\gamma e^{tA}x\|_1^2
=\Big\|\sum_{n\in\mN}\alpha_n^{\gamma}e^{-\alpha_nt}\<x,e_{1,n}\>_1e_{1,n}\Big\|_1^2\\
&=\sum_{n\in\mN} \alpha_n^{2(\gamma-\theta)}e^{-2\alpha_nt}\alpha_n^{2\theta}|\<x,e_{1,n}\>_1|^2\\
&\leq C_{\gamma,\theta} t^{-2(\gamma-\theta)}e^{-\alpha_1t/2}\sum_{n\in\mN} \alpha_n^{2\theta}|\<x,e_{1,n}\>_1|^2=C_{\theta,\gamma} t^{-2(\gamma-\theta)}e^{-\alpha_1t/2}\|x\|_{(-A)^\theta}^2,
\end{align*}
which yields (\ref{pp1}). To show estimate (\ref{pp2}), by the basic inequality that $1-e^{-\alpha t}\leq C_{\alpha,\gamma} t^\gamma$ with $\gamma\in[0,1],\alpha>0$, we deduce that for any $0\leq s\leq t$ and $x\in H_1$,
\begin{align*}
\|e^{tA}x-e^{sA}x\|_1^2
&=\Big\|\sum_{n\in\mN}(e^{-\alpha_nt}-e^{-\alpha_ns})\<x,e_{1,n}\>_1e_{1,n}\Big\|_1^2\\
&\leq C_\gamma\sum_{n\in\mN} \alpha_n^{2\gamma}(t-s)^{2\gamma}e^{-2\alpha_ns}|\<x,e_{1,n}\>_1|^2\\
&=C_\gamma(t-s)^{2\gamma}|\<(-A)^\gamma e^{sA}x,e_{1,n}\>_1|^2\\
&=C_\gamma(t-s)^{2\gamma}\| e^{sA}x\|_{(-A)^\gamma}^2.
\end{align*}
Combining (\ref{pp1}) and (\ref{pp2}), we  immediately get (\ref{pp3}).
\end{proof}

For $i=1,2$, let $Q_i$ be two linear self-adjoint operators on $H_i$ with positive eigenvalues $\{\lambda_{i,n}\}_{n\in\mN}$, i.e.,
$$
Q_ie_{i,n}=\lambda_{i,n}e_{i,n},\quad\forall n\in\mN.
$$
Let $W_t^i,i=1,2$, be $H_i$-valued $Q_i$-Wiener processes both defined on a complete filtered probability space $(\Omega,\mathscr{F},\mathscr{F}_t,\mP)$. Then it is known that $W_t^i$ can be written as
$$
W_t^i=\sum_{n\in\mN}\sqrt{\lambda_{i,n}}\beta_{i,n}(t)e_{i,n},
$$
where $\{\beta_{i,n}\}_{n\in\mN}$ are mutual independent real-valued Brownian motions. Note that $W_t^i$ ($i=1,2$) are non-degenerate. We shall further assume that:

 \vspace{2mm}
\noindent{\bf (A2):} For $i=1,2$,
$$
Tr(Q_i):=\sum_{n\in\mN}\lambda_{i,n}<+\infty\quad\text{and}\quad Tr((-A)Q_1)<+\infty,
$$
and for any $T>0$,
\begin{align*}
\int_0^T\Upsilon_t^{\frac{1+\theta}{2}}\dif t<\infty,
\end{align*}
where 
\begin{align*}
\Upsilon_t:=\sup\limits_{n\geq 1}\frac{2\beta_n}{\lambda_{2,n}(e^{2\beta_nt}-1)}<\infty,
\end{align*}
$\beta_n$ is given by (\ref{bnbn}), and $\theta\geq\max{(\eta,1-\eta)}$ with $\eta$ being the H\"older regularity of the coefficients in the assumptions of Theorem 2.2 below (see \cite[Lemma 9]{DF} for this condition).

\subsection{Main results}

The first main result of this paper is about the strong convergence in the averaging principle for SPDE (\ref{spde1}).

\bt[Strong convergence]\label{main1}
Let $T>0$, $x\in \cD((-A)^\theta)$ and $y\in \cD((-B)^\theta)$ with $\theta>0$. Assume that  ({\bf A1}) and ({\bf A2})  hold, $F\in C^{2,\eta}_p(H_1\times H_2,H_1)$ and $G\in C^{2,\eta}_b(H_1\times H_2,H_2)$ with $\eta>0$. Then for any  $q\geq 1$ and $\gamma\in[0,\theta\wedge1/2)$,  we have
	\begin{align}\label{rat1}
	\sup\limits_{t\in[0,T]}\mE\|X^{\eps}_t-\bar X_t\|_{(-A)^\gamma}^q\leq C_1\,\eps^{\frac{q}{2}},
	\end{align}
	where $\bar X_t$ solves equation (\ref{spde2}), and  $C_1=C(T,x,y)>0$ is a constant independent of $\eta$ and $\eps$.
	\et

To compare our result with previous work in the literature, we make the following comments:

\br\label{r1}
(i) When $\gamma=0$ in (\ref{rat1}), the $1/2$-order rate of convergence in the $L^2(\Omega)$-sense is known to be optimal, which is the same as in the SDE case. However, the convergence in $\|\cdot\|_{(-A)^\gamma}$ norm seems to have never been studied before. This is particularly interesting for SPDEs since $A$ is in general an unbounded operator, and will play an important role below to study the homogenization for $Z_t^\eps$.

\vspace{1mm}
(ii) Note that the coefficients are assumed to be only $\eta$-H\"older continuous with respect to the fast variable, and the convergence rate does not dependent on $\eta$. This indicates that the convergence in the averaging principle does not depend on the regularity of the coefficients with respect to the fast variable, which coincides  with the intuition, since in the limit equation the fast component has been totally averaged out.

\vspace{1mm}
(iii) Let us explain why $\gamma<1/2$ should be the best possible. In fact, the main reason is that the processes $X_t^\eps$ and $Y_t^\eps$ are only $\gamma$-H\"older continuous with respect to the time variable with $\gamma<1/2$. From another point of view, for $Z_t^\eps$ given by (\ref{zt}), estimate (\ref{rat1}) means that for every $t\geq 0$, we have
 $$
\sup_{\eps\in(0,1)} \mE\|(-A)^\gamma Z_t^\eps\|_1^2<\infty.
 $$
But by Theorem \ref{main3} below, we have that $Z_t^\eps$ converges to $\bar Z_t$ with $\bar Z_t$ satisfying (\ref{spdez}). Through straightforward computations we find that  $\mE\|(-A)^\gamma \bar Z_t\|_1^2<\infty$ only when $\gamma<1/2$.
\er

Recall that $Z_t^\eps$ is defined by (\ref{zt}).
To study the homogenization for $Z_t^\eps$, we need to consider the following Poisson equation:
\begin{align}\label{poF}
\cL_2(x,y)\Psi(x,y)=-\delta F(x,y),
\end{align}
where $\delta F$ is given by (\ref{dF}), and $\cL_2(x,y)$ is defined by
\begin{align}\label{L2}
\cL_2\varphi(x,y)&:=\cL_2(x,y)\varphi(x,y):=\langle By+G(x,y), D_y\varphi(x,y)\rangle_2\no\\&
\quad+\frac{1}{2}Tr\big[D^2_{y}\varphi(x,y)Q_2\big],
\quad\forall\varphi\in C_p^{0,2}(H_1\times H_2).
\end{align}
According to Theorem \ref{PP} and Remark \ref{PPR} below, there exists a unique solution $\Psi$ to equation (\ref{poF}).
It turns out that the limit  $\bar Z_t$ of $Z_t^\eps$  satisfies the following linear equation:
\begin{align}\label{spdez}
\dif \bar Z_t=A\bar Z_t\dif t+D_x\bar F(\bar X_t).\bar Z_t\dif t+\sigma(\bar X_t)\dif \tilde W_t,\quad \bar Z_0=0,
\end{align}
where   $\tilde W_t$ is a cylindrical  Wiener process in $H_1$ which is independent of $W_t^1$, and $\sigma:H_1\to \sL(H_1)$ satisfies
\begin{align}\label{sst}
\frac{1}{2}\sigma(x)\sigma^*(x)=\overline{\delta F\otimes\Psi}(x):=\int_{H_2}\big[\delta F(x,y)\otimes\Psi(x,y)\big]\mu^x(\dif y).
\end{align}

The following is the second main result of this paper.

\bt[Normal deviations]\label{main3}
 Let $T>0$, $x\in \cD((-A)^\theta)$ and $y\in \cD((-B)^\theta)$ with $\theta>0$. Assume that  ({\bf A1}) and ({\bf A2}) hold,  $F\in C^{2,\eta}_p(H_1\times H_2,H_1)$ and $G\in C^{2,\eta}_b(H_1\times H_2,H_2)$ with $\eta>0$. Then for any $\gamma\in(0,1/2)$   and $\varphi\in \mC_b^4(H_1)$,   we have
\begin{align*}
\sup_{t\in[0,T]}\big|\mE[\varphi(Z_t^{\eps})]-\mE[\varphi(\bar Z_{t})]\big|\leq C_2\,\eps^{\frac{1}{2}-\gamma},
\end{align*}
where  $C_2=C(T,x,y,\varphi)>0$ is a constant independent of $\eta$ and $\eps$.
\et

\br
Note that we claim that $\tilde W_t$ in (\ref{spdez}) is independent of $W_t^1$. The advantage of formula (\ref{sst}) is that we can study the regularity properties of $\sigma$ directly by using the result of the Poisson equation established in Theorem \ref{PP} below. Furthermore, one can check that $\sigma(x)$ is a Hilbert-Schmidt operator.  In fact, by Theorem \ref{PP} we have
\begin{align*}
\|\sigma(x)\|_{\sL_2(H_1)}^2
&=\sum_{n\in\mN}\<\sigma(x)\sigma^*(x)e_{1,n},e_{1,n}\>_1\no\\
&=2\sum_{n\in\mN}\<\int_{H_2}\big[\delta F(x,y)\otimes\Psi(x,y)\big]\mu^x(\dif y)e_{1,n},e_{1,n}\>_1\no\\
&=2\int_{H_2}\<\delta F(x,y),\Psi(x,y)\>_1\mu^x(\dif y)\no\\
&\leq C_0\int_{H_2}(1+\|y\|_2^{2p})\mu^x(\dif y)< \infty.
\end{align*}
Thus, the stochastic integral part in (\ref{spdez}) is well-defined.
\er

\section{Poisson Equation in Hilbert space}

Consider the following Poisson equation in the infinite dimensional Hilbert space $H_2$:
\begin{align}\label{pois}
\cL_2(x,y)\psi(x,y)=-\phi(x,y),
\end{align}
where $\cL_2(x,y)$ is defined by (\ref{L2}), $x\in H_1$ is regarded as a parameter, and $\phi: H_1\times H_2\rightarrow \mR$ is a Borel-measurable function.
Recall that $Y_t^x(y)$  satisfies the frozen equation (\ref{froz}) and $\mu^x(\dif y)$ is the   invariant measure of  $Y_t^x(y)$ (see Lemma \ref{long} below).
Since we are considering (\ref{pois}) on the whole space and not on a compact subset, it is necessary to make the following ``centering" assumption  on $\phi$:
\begin{align}\label{cen}
\int_{H_2}\phi(x,y)\mu^x(\dif y)=0,\quad\forall x\in H_1.
\end{align}
Such kind of assumption is also  natural and analogous to the centering condition in the standard central limit theorem, see e.g. \cite{P-V,P-V2}.

We first introduce the following definition of solutions for equation (\ref{pois}).

\bd
A measurable function $\psi:  H_1\times H_2\rightarrow \mR$ is said to be a classical solution to  equation (\ref{pois}) if:

\vspace{2mm}
\noindent
(i) the function $\psi(x,y)\in C^{0,2}_p(H_1\times H_2)$ and for any $(x,y)\in H_1\times H_2$, the operator $D^2_y\psi(x,y)\in \sL(H_2)$;

\vspace{1mm}
\noindent
(ii) for any $x\in H_1$ and $y\in \cD(B)$, the function $\psi$ satisfies equation (\ref{pois}).
\ed

The main aim of this section is to prove the following result.

\bt\label{PP}
Let $\eta>0$ and $k=0,1,2$. Assume that {\bf (A1)} and  {\bf (A2)}  hold, and $G\in C_b^{k,\eta}(H_1\times H_2, H_2)$. Then for every $\phi\in C^{k,\eta}_p(H_1\times H_2)$ satisfying (\ref{cen}),
there exists  a  unique classical solution $\psi\in C^{k,0}_p(H_1\times H_2)\cap \mC^{0,2}_p(H_1\times H_2)$ to equation (\ref{pois}) satisfying (\ref{cen}), which is given by
\begin{align}\label{pro}
\psi(x,y)=\int_0^\infty\!\mE\big[\phi(x,Y_t^x(y))\big]\dif t,
\end{align}
where $Y_t^x(y)$ satisfies the frozen equation (\ref{froz}).
\et

\br\label{PPR}
We can also solve the Poisson equation (\ref{pois}) for  Hilbert space valued function $\tilde\phi$, i.e., $\tilde\phi: H_1\times H_2\to H$ with $H$ being another Hilbert space. In fact, let $\{e_n\}_{n\in\mN}$ be the orthonormal basis of $H$, and define
$$
\phi_n(x,y):=\<\tilde\phi(x,y),e_n\>_H.
$$
Then for each $n\in\mN$, we have $\phi_n: H_1\times H_2\to \mR$, and thus there exists  a  solution $\psi_n: H_1\times H_2\to \mR$ to the equation (\ref{pois}) with $\phi$ replaced by $\phi_n$. Define a $H$-valued function by
$$
\tilde\psi(x,y):=\sum_{n\in\mN}\psi_n(x,y)e_n=\int_0^\infty\mE[\tilde\phi(x,Y_t^x(y))]\dif t.
$$
Then $\tilde\psi$ solves
\begin{align*}
\cL_2(x,y)\tilde\psi(x,y)=-\tilde\phi(x,y).
\end{align*}

\er
\subsection{Properties of the frozen transition semigroup}

Given $\phi: H_1\times H_2\to \mR$, let
$$
T_t\phi(x,y):=\mE\big[\phi(x,Y_t^x(y))\big],
$$
In view of (\ref{pro}), we need to study the behavior of $T_t\phi$ as well as its first and second order derivatives with respect to the $y$ variable both near $t =0$ and as $t\to\infty$.
Let us first collect the following estimates for $Y_t^x(y)$.

\bl\label{long}
Assume {\bf (A1)} and {\bf (A2)}   hold, and that $G\in C_b^{0,\eta}(H_1\times H_2, H_2)$. Then there exists a unique mild solution $Y_t^x(y)$ to the equation (\ref{froz}). Moreover, we have:

\vspace{2mm}
\noindent
(i) For any $t\geq 0$ and $q\geq 1$, there exist  constants $C_q, \lambda>0$ such that
\begin{align}\label{ytxy}
\mE\|Y_{t}^{x}(y)\|_2^q\leq C_q\big(1+e^{-\lambda  t}\|y\|_2^q\big);
\end{align}
(ii) $Y_t^x(y)$ is strong Feller and irreducible;

\vspace{1mm}
\noindent
(iii) There exist constants  $C_0,\lambda>0$ such that for any $t\geq 0$ and every $\phi\in L^\infty_p(H_1\times H_2)$,
\begin{align}\label{exp}
\left\vert T_t\phi(x,y)-\int_{H_2}\phi(x,z)\mu^x(\dif z)\right\vert\leq C_0\|\phi\|_{L^\infty_p}(1+\|x\|_1+\|y\|_2^p)e^{-\lambda t}.
\end{align}
\el

\begin{proof}

The existence and uniqueness of solutions to SPDE (\ref{froz}) with  H\"older continuous coefficients follows from \cite[Theorem 7]{DF}. We only need to verify that the assumptions 4, 5, 6 in \cite{DF} hold. To this end, let
$$
Q_t:=\int_0^te^{sB}Q_2e^{sB^*}\dif s\quad \text{and} \quad\Lambda_t=Q_t^{-1/2}e^{tB}.
$$
Then under the assumption ({\bf A2}) we have
$$
Tr(Q_t)=\sum\limits_{n\in\mN}\frac{\lambda_{2,n}}{2\beta_n}(1-e^{-2\beta_nt})
\leq\sum\limits_{n\in\mN}\frac{\lambda_{2,n}}{2\beta_1}\leq C_0Tr(Q_2)<+\infty.
$$
Note that
$$
\|\Lambda_t\|_{\sL(H_2)}^2=\sup_{n\in\mN}\frac{2\beta_n}{\lambda_{2,n}(e^{2\beta_nt}-1)}.
$$
Thus, we have 
\begin{align*}
\int_0^T\|\Lambda_t\|_{\sL(H_2)}^{1+\theta}\dif t<\infty,\quad\text{for}\; \text{some}\; \theta\geq \max{(\eta,1-\eta)},
\end{align*}
which implies the desired result. Meanwhile, estimate (\ref{ytxy}) can be proved by following the same argument as in \cite[Theorem 7.3]{Ce4}, and the conclusions in (ii) follow by \cite[Theorem 4 and Proposition 4]{C-G}. Furthermore, for any $t\in [0,T]$ one can check that there exists a $\theta>0$ such that
$$
{\mE}\|Y^{x}_t(y)\|_{(-B)^\theta}\leq C_T\big(1+\|y\|_2^p\big).
$$
For any $r, R>0$, let $B_r:=\{y\in H_2: \|y\|_2\leq r\}$ and $K=\{y\in H_2: \|y\|_{(-B)^\theta}\leq R\}$. Then we have that for $R$ large enough,
\begin{align*}
\inf_{y\in B_r}{\mP}(Y^{x}_{T}(y)\in K)&=\inf_{y\in B_r}{\mP}(\|Y^{x}_{T}(y)\|_{(-B)^\theta}\leq R)\\
&=1-\sup_{y\in B_r}{\mP}(\|Y^{x}_{T}(y)\|_{(-B)^\theta}>R)\\
&\geq 1-\sup_{y\in B_r}\frac{{\mE}\|Y^{x,y}_T\|_{(-B)^\theta}}{R}
\geq 1-\frac{C_T(1+r^p)}{R}>0.
\end{align*}
	Thus, estimate (\ref{exp}) follows by  \cite[Theorem 2.5]{GM}.
\end{proof}

Let $P_t$ be the Ornstein-Uhlenbeck semigroup  defined by
$$
P_t\phi(x,y):=\mE\big[\phi(x,R_t(y))\big],
$$
where
$$
\dif R_t=BR_t\dif t+\dif W_t^2,\quad R_0=y\in H_2.
$$
The following result was proved by \cite[Theorem 4]{DF} if $\phi$ is bounded and measurable. In view of (\ref{ytxy}), we can generalize it to   $\phi$ with polynomial growth  by following exactly the same argument. We omit the details here.

\bl\label{pt}
Assume {\bf (A1)} and  {\bf (A2)}   hold. Then for every $\phi\in L^\infty_p(H_1\times H_2)$ and $t\in(0,T]$, we have $P_t\phi(x,y)\in \mC_p^{0,2}(H_1\times H_2)$. Moreover,
\begin{align}\label{rty}
\|D_yP_t\phi(x,y)\|_2\leq C_T\frac{1}{\sqrt{t}}\|\phi\|_{L^\infty_p}(1+\|x\|_1+\|y\|_2^p),
\end{align}
and for any $\eta\in[0,1]$,
\begin{align}\label{rty2}
\|D^2_yP_t\phi(x,y)\|_{\sL(H_2)}\leq C_T\frac{1}{t^{1-\eta/2}}\|\phi\|_{C^{0,\eta}_p}(1+\|x\|_1+\|y\|_2^p),
\end{align}
	where  $C_T>0$ is a  constant.
\el

Based on Lemma \ref{pt}, we  have the following result.

\bl
Assume {\bf (A1)} and {\bf (A2)}  hold, and that $G\in C_b^{0,\eta}(H_1\times H_2, H_2)$. Then for every $\phi\in L^\infty_p(H_1\times H_2)$ satisfying (\ref{cen}),  we have $T_t\phi(x,y)\in \mC_p^{0,1}(H_1\times H_2)$ with
\begin{align}\label{tt}
|T_t\phi(x,y)|\leq C_0\|\phi\|_{L^\infty_p}(1+\|x\|_1+\|y\|_2^p)e^{-\lambda t}
\end{align}
and
\begin{align}\label{tty}
\|D_yT_t\phi(x,y)\|_2\leq C_0\frac{1}{\sqrt{t}\wedge1}\|\phi\|_{L^\infty_p}(1+\|x\|_1+\|y\|_2^p)e^{-\lambda t},
\end{align}
where  $C_0, \lambda>0$ are  constants independent of $t$. If we further assume that $\phi\in C^{0,\eta}_p(H_1\times H_2)$ with $\eta\in(0,1)$, then   $T_t\phi(x,y)\in \mC_p^{0,2}(H_1\times H_2)$ and
\begin{align}\label{tty2}
\|D^2_yT_t\phi(x,y)\|_{\sL(H_2)}\leq C_0\frac{1}{t^{1-\eta/2}\wedge1}\|\phi\|_{C^{0,\eta}_p}(1+\|x\|_1+\|y\|_2^p)e^{-\lambda t}.
\end{align}
\el
\begin{proof}
Estimate (\ref{tt}) follows by (\ref{exp}) directly. The assertions that $T_t\phi(x,y)\in \mC_p^{0,1}(H_1\times H_2)$ and $T_t\phi(x,y)\in \mC_p^{0,2}(H_1\times H_2)$ can be obtained as in \cite[Theorem 5]{DF}. Let us focus on the a-priori estimates  (\ref{tty}) and (\ref{tty2}).
By Duhamel's formula (see e.g. \cite[(16)]{DF}), for any $t>0$ we have
\begin{align*}
T_t\phi(x,y)=P_t\phi(x,y)+\int_0^tP_{t-s}\<G,D_yT_s\phi\>_2(x,y)\dif s.
\end{align*}
In view of (\ref{rty}) and by the assumption that $G$ is bounded, we have for every $t\in(0,T]$,
\begin{align*}
\|D_yT_t\phi(x,y)\|_2\leq C_0(1+\|x\|_1&+\|y\|_2^p)\bigg(\frac{1}{\sqrt{t}}\|\phi\|_{L^\infty_p}\\
&+\int_0^t\frac{1}{\sqrt{t-s}}\|D_yT_s\phi(x,y)\|_2\dif s\bigg).
\end{align*}
By Gronwall's inequality  we  obtain
\begin{align}\label{es11}
\|D_yT_t\phi(x,y)\|_2\leq C_0\frac{1}{\sqrt{t}}\|\phi\|_{L^\infty_p}(1+\|x\|_1+\|y\|_2^p),
\end{align}
which means that (\ref{tty}) is true for $t\leq 2$. For $t>2$, by the Markov property we have
$$
T_t\phi(x,y)=\mE \big[T_{t-1}\phi(x, Y_1^x(y))\big].
$$
Using (\ref{tt}) and  (\ref{es11}) with $t=1$ and $\phi$ replaced by $T_{t-1}\phi$, we deduce that
\begin{align*}
\|D_yT_t\phi(x,y)\|_2&\leq C_1\|T_{t-1}\phi\|_{L^\infty_p}(1+\|x\|_1+\|y\|_2^p)\\
&\leq C_1e^{-\lambda(t-1)}\|\phi\|_{L^\infty_p}(1+\|x\|_1+\|y\|_2^p).
\end{align*}
To prove (\ref{tty2}), we first note that by (\ref{rty}), (\ref{rty2}) and interpolation, we have that for any $\eta\in(0,1)$,
\begin{align*}
\|D_yP_t\phi\|_{C_p^{0,\eta}}\leq C_0\frac{1}{t^{(1+\eta)/2}}\|\phi\|_{L^\infty_p}.
\end{align*}
Thus, we derive that
\begin{align*}
\|D_yT_t\phi\|_{C_p^{0,\eta}}&\leq C_0\bigg(\frac{1}{t^{(1+\eta)/2}}\|\phi\|_{L^\infty_p}+\int_0^t\frac{1}{(t-s)^{(1+\eta)/2}}\|D_yT_s\phi\|_{L_p^{\infty}}\dif s\bigg)\\
&\leq C_0\frac{1}{t^{(1+\eta)/2}\wedge1}\|\phi\|_{L^\infty_p}.
\end{align*}
Combining this with (\ref{rty2}) and the assumption that $G\in C_b^{0,\eta}(H_1\times H_2)$, we get for $k_1,k_2\in H_2$,
\begin{align*}
|D^2_yT_t\phi(x,y).(k_1,k_2)|\leq C_2\frac{1}{t^{(2-\eta)/2}\wedge1}\|\phi\|_{C^{0,\eta}_p}(1+\|x\|_1+\|y\|_2^p)\|k_1\|_2\|k_2\|_2,
\end{align*}
which means that (\ref{tty2}) holds for $t\leq 2$. Following the same ideas as above, we obtain that (\ref{tty2}) holds for $t>2$.
\end{proof}

\subsection{Proof of Theorem \ref{PP}}

\begin{proof}[Proof of Theorem \ref{PP}]
We divide the proof into three steps.

\vspace{1mm}
\noindent{\bf Step 1.}
Let $\psi(x,y)$ be defined by (\ref{pro}). We first prove that $\psi\in \mC_p^{0,2}(H_1\times H_2)$. In fact, for every $\phi\in L^\infty_p(H_1\times H_2)$ satisfying (\ref{cen}), by (\ref{tt}) we deduce that
\begin{align}\label{dydy1}
|\psi(x,y)|&\leq \int_0^\infty |T_t\phi(x,y)|\dif t\leq C_0\|\phi\|_{L^\infty_p}\int_0^\infty (1+\|x\|_1+\|y\|_2^p)e^{-\lambda t}\dif t\no\\
&\leq C_0\|\phi\|_{L^\infty_p}(1+\|x\|_1+\|y\|_2^p),
\end{align}
and by (\ref{tty}) we have for every $k_1\in H_2$,
\begin{align}
|\<D_y\psi(x,y),k_1\>_2|&\leq \int_0^\infty |\<D_yT_t\phi(x,y),k_1\>_2|\dif t\no\\
&\leq C_1\|\phi\|_{L^\infty_p}\int_0^\infty \frac{1}{\sqrt{t}\wedge1}(1+\|x\|_1+\|y\|_2^p)\|k_1\|_2\,e^{-\lambda t}\dif t\no\\
&\leq C_1\|\phi\|_{L^\infty_p}(1+\|x\|_1+\|y\|_2^p)\|k_1\|_2.\label{dydy}
\end{align}
Furthermore, by the dominated convergence theorem we deduce that
\begin{align*}
&\lim_{\|k_1\|_2\to0}\frac{|\psi(x,y+k_1)-\psi(x,y)-\<D_y\psi(x,y),k_1\>_2|}{\|k_1\|_2}\\
&\leq\lim_{\|k_1\|_2\to0}\int_0^\infty \frac{|T_t\phi(x,y+k_1)-T_t\phi(x,y)-\<D_yT_t\phi(x,y),k_1\>_2|}{\|k_1\|_2}\dif t=0.
\end{align*}
Similarly, by using (\ref{tty2})  we can prove that $\psi\in \mC_p^{0,2}(H_1\times H_2)$ and for every $k_1, k_2\in H_2$,
\begin{align}\label{dydy2}
|D^2_y\psi(x,y).(k_1,k_2)|\leq C_2\|\phi\|_{C^{0,\eta}_p}(1+\|x\|_1+\|y\|_2^p)\|k_1\|_2\|k_2\|_2.
\end{align}
Here, we remark that the control of $\psi$ and $D_y\psi$ depends only on the $\|\cdot\|_{L_p^\infty}$-norm of the function $\phi$. In addition, by Fubini's theorem and the property of the invariant measure, we have
\begin{align*}
\int_{H_2}\psi(x,y)\mu^x(\dif y)&=\int_0^\infty\!\! \int_{H_2}T_t\phi(x,y)\mu^x(\dif y)\dif t\\
&=\int_0^\infty\!\! \int_{H_2}\phi(x,y)\mu^x(\dif y)\dif t=0.
\end{align*}
Thus, the assertion  that $\psi$ is the unique solution for equation (\ref{pois}) follows by It\^o's formula, see e.g. \cite[Lemma 4.3]{Br1}.

\vspace{1mm}

\noindent{\bf Step 2.} When $k=1$ in the assumptions, we  prove  that  $\psi(x,y)\in C^{1,0}_p(H_1\times H_2)$. In fact, for every $h_1\in H_1$ and $\tau>0$,  we have
\begin{align*}
\cL_2(x,y)&\frac{\psi(x+\tau h_1,y)-\psi(x,y)}{\tau}=-\frac{\phi(x+\tau h_1,y)-\phi(x,y)}{\tau}\\
&-\frac{\<G(x+\tau h_1,y)-G(x,y),D_y\psi(x+\tau h_1,y)\>_2}{\tau}=:-\phi_{\tau,h_1}(x,y).
\end{align*}
By the assumptions on $\phi$ and $G$, and using estimates (\ref{dydy}) and (\ref{dydy2}), one can check that $\phi_{\tau,h_1}(x,y)\in C_p^{0,\eta}(H_1\times H_2)$. We claim that
\begin{align}\label{ccen}
\int_{H_2}\phi_{\tau,h_1}(x,y)\mu^x(\dif y)=0,\quad\forall \tau>0, x, h_1\in H_1.
\end{align}
Then, according to Step 1, we obtain that for every $\tau>0$,
\begin{align}\label{diff}
\frac{\psi(x+\tau h_1,y)-\psi(x,y)}{\tau}=\int_0^\infty\mE\phi_{\tau,h_1}(x,Y_t^{x}(y))\dif t.
\end{align}
Note that
\begin{align*}
\lim_{\tau\to0}\phi_{\tau,h_1}(x,y)=\<D_x\phi(x,y),h_1\>_1+\langle D_x G(x,y). h_1,D_y\psi(x,y)\rangle_2=:\phi_{h_1}(x,y).
\end{align*}
Using the assumption that $G\in C_b^{1,0}(H_1\times H_2)$ and (\ref{dydy}), we find that
\begin{align}\label{dom6}
\vert \phi_{h_1}(x,y)\vert&\leq \|D_x\phi(x,y)\|_1\|h_1\|_1+\|D_x G(x,y)\|_{\sL(H_1,H_2)}\|h_1\|_1\|D_y\psi(x,y)\|_2\no\\
&\leq C_3(1+\|x\|_1+\|y\|_2^p)\|h_1\|_1.
\end{align}
Thus, by the dominated convergence theorem, we obtain
\begin{align*}
\int_{H_2}\phi_{h_1}(x,y)\mu^x(\dif y)=0,\quad\forall x, h_1\in H_1.
\end{align*}
Combining this with (\ref{tt}), we have
\begin{align*}
|T_t\phi_{\tau,h_1}(x,y)|&\leq C_4\|\phi_{\tau,h_1}\|_{L^\infty_p}(1+\|x\|_1+\|y\|_2^p)e^{-\lambda t}\\
&\leq C_4\big(1+\|\phi_{h_1}\|_{L^\infty_p}\big)(1+\|x\|_1+\|y\|_2^p)e^{-\lambda t}.
\end{align*}
As a result, taking the limit $\tau\to 0$ on both sides of (\ref{diff}) we get
\begin{align}\label{dx1}
\<D_x\psi(x,y),h_1\>_1=\int_0^\infty\mE\phi_{h_1}(x,Y_s^{x}(y))\dif s.
\end{align}

It remains to prove (\ref{ccen}). To this end, for every $n\in\mN$, let $H_2^n:= span\{e_{2,k};1\leq k\leq n\}$ and denote the orthogonal projection of $H_2$ onto $H_2^n$ by $P_{2}^n$. We introduce the following approximation of system (\ref{froz}):
\begin{align}\label{froz2}
\dif Y^{x,n}_t =B_nY^{x,n}_t\dif t+ G_n(x, Y^{x,n}_t)\dif t+ P_{2}^n\dif W_t^2,\;\;Y^{x,n}_0=P_2^ny\in H_2^n.
\end{align}
 where for $(x,y)\in H_1\times H_2$,
\begin{align}\label{gn}
B_ny:=BP_2^ny\quad\text{and}\quad G_n(x,y):=P_{2}^nG(x,P_2^ny).
\end{align}
It is easy to check that  $G_n$ is uniformly bounded with respect to $n$. Thus the solution to equation (\ref{froz2}) has the same long-time behavior as the one to equation (\ref{froz}). Let $\mu^x_n(\dif y)$ be the invariant measure for $Y_t^{x,n}$, and define $\cL^n_2(x,y)$  by
\begin{align*}
\cL^n_2(x,y)\varphi(x,y)&:=\<B_ny+ G_n(x,y),D_y\varphi(x,y))\>_2\\
&\quad\,\,+\frac{1}{2}Tr\big[D_y^2\varphi(x,y)Q_{2,n}\big],\quad\forall\varphi\in C^{0,2}_p(H_1\times H_2^n),
\end{align*}
where $Q_{2,n}:=Q_2P_2^n$.
Consider the Poisson equation corresponding to (\ref{froz2}):
\begin{align}\label{pois22}
\cL_2^n(x,y)\psi^n(x,y)=-\phi(x,P_2^ny)=:-\phi^n(x,y).
\end{align}
As in Step 1, the unique solution is given by
\begin{align*}
\psi^n(x,y)=\int_0^\infty\!\mE\big[\phi^n(x,Y_t^{x,n}(y))\big]\dif t.
\end{align*}
According to \cite[Subsection 4.1]{Br1} (see also \cite[Section 6]{CF}), we have  for every $x\in H_1$ and $y\in H_2$,
\begin{align}\label{conv}
\lim_{n\to\infty}\psi^n(x,y)=\psi(x,y),\quad\lim_{n\to\infty}\<D_y\psi^n(x,y),k_2\>_2=\<D_y\psi(x,y),k_2\>_2.
\end{align}
For every $\tau>0$ and $h_1\in H_1$, define
\begin{align*}
\phi^n_{\tau,h_1}(x,y)&:=\big[\phi^n(x+\tau h_1,y)-\phi^n(x,y)\big]\\
&\quad\,\,+\<G_n(x+\tau h_1,y)-G_n(x,y),D_y\psi^n(x+\tau h_1,y)\>_2.
\end{align*}
Since (\ref{pois22}) is an equation in finite dimensions, according to \cite[Lemma 3.2]{RX} we have
\begin{align*}
\int_{H_2^n} \phi^n_{\tau,h_1}(x,y) \mu^x_n(\dif y)=0.
\end{align*}
Using estimates (\ref{dom6}), (\ref{conv}), the formula above \cite[(4.4)]{Br1} and taking the limit $n\to \infty$ on both sides of the above equality, we obtain (\ref{ccen}).

\vspace{1mm}
\noindent{\bf Step 3.} When $k=2$ in the assumptions, we prove  that  $\psi(x,y)\in C^{2,0}_p(H_1\times H_2)$. In this case, we mainly focus on the a-priori estimate. The specific procedure can be done as in Step 2. In view of (\ref{dx1}) and according to the results in Step 1, we can conclude that $\<D_x\psi(x,y),h_1\>_1\in \mC_p^{0,2}(H_1\times H_2)$ with
\begin{align*}
|\<D_x\psi(x,y),h_1\>_1|&\leq C_5\|\phi_{h_1}\|_{L^\infty_p}(1+\|x\|_1+\|y\|_2^p)\\
&\leq C_5(1+\|x\|_1+\|y\|_2^p)\|h_1\|_1,
\end{align*}
and
\begin{align}\label{dyx}
|D_yD_x\psi(x,y).(h_1,k)|&\leq C_5\|\phi_{h_1}\|_{L^\infty_p}(1+\|x\|_1+\|y\|_2^p)\|k\|_2\no\\
&\leq C_5 (1+\|x\|_1+\|y\|_2^p)\|h_1\|_1\|k\|_2.
\end{align}
Moreover, we have
\begin{align}\label{d1}
{\mathcal{L}_2}(x,y)\<D_x\psi(x,y),h_1\>_1=-\phi_{h_1}(x,y),
\end{align}
Since $\<D_x\psi(x,y),h_1\>_1$ is a classical solution,
by taking derivative with respect to the $x$ variable on both sides of the equation, we have that for any $h_1, h_2\in H_1$,
\begin{align*}
&{\mathcal{L}_2}(x,y)(D^2_{x}\psi(x,y).(h_1,h_2))=-D^2_{x}\phi(x,y).(h_1,h_2)\\
&\qquad-2D_yD_x\psi(x,y).(h_2,D_x G(x,y).h_1)\\
&\qquad-\langle D^2_{x}G(x,y).(h_1,h_2),D_y\psi(x,y)\rangle_2=:-\phi_{h_1,h_2}(x,y).
\end{align*}
By the assumption that $G\in C_b^{2,0}(H_1\times H_2)$ and (\ref{dyx}), we get
\begin{align*}
\vert \phi_{h_1,h_2}(x,y)\vert&\leq\|D^2_{x}\phi(x,y)\|_{\sL(H_1\times H_1)}\|h_1\|_1\|h_2\|_1\\
&\quad+2\vert D_yD_x\psi(x,y).(h_2,D_xG(x,y).h_1)\vert\\
&\quad+\vert\langle D^2_{x}G(x,y).(h_1,h_2),D_y\psi(x,y)\rangle_2\vert\\
&\leq C_6(1+\|x\|_1+\|y\|_2^p)\|h_1\|_1\|h_2\|_1.
\end{align*}
Furthermore,  by using \cite[Lemma 3.2]{RX} again and the same approximation argument as in Step 2, we have that $\phi_{h_1,h_2}(x,y)$ satisfies the centering condition
\begin{align}\label{c2}
\int_{H_2}\phi_{h_1,h_2}(x,y)\mu^x(\dif y)=0.
\end{align}
Thus, in view of (\ref{dydy1}) and (\ref{dydy}) we can get that $(D^2_{x}\psi(x,y).(h_1,h_2))\in \mC_p^{0,2}(H_1\times H_2)$ with
\begin{align*}
|D^2_{x}\psi(x,y).(h_1,h_2)|&\leq C_7\|\phi_{h_1,h_2}\|_{L^\infty_p}(1+\|x\|_1+\|y\|_2^p)\\
&\leq C_7(1+\|x\|_1+\|y\|_2^p)\|h_1\|_1\|h_2\|_1.
\end{align*}
The proof is finished.
\end{proof}

Given a function $F(x,y)$, recall that $\bar{F}$ is defined by (\ref{df1}). Usually, it is not so easy to study the regularity of the averaged function, which contains a separate problem connected with the smoothness of the invariant measure $\mu^x(\dif y)$.  Here, we provide a simple method by using Theorem \ref{PP}.

\bl\label{aveF}
Assume that $F\in C_p^{1,\eta}(H_1\times H_2)$ with $\eta>0$, and let $\Psi(x,y)$ solve the Poisson equation (\ref{poF}). Then for any $h_1\in H_1$, we have
\begin{align}\label{trans1}
D_x\bar F(x).h_1=\!\int_{H_2}\!\!\Big[D_x F(x,y).h_1\!+\!\<D_xG(x,y).h_1,D_y\Psi(x,y)\>_2\Big]\mu^x(\dif y).
\end{align}
Furthermore, assume that $F\in C_p^{2,\eta}(H_1\times H_2)$, then we have for any $h_1,h_2\in H_1$,
\begin{align}\label{trans2}
D^2_x\bar F(x).(h_1,h_2)&=\!\int_{H_2}\!\!\Big[D_x^{2} F(x,y).(h_1,h_2)\!+\!2D_yD_x\Psi(x,y).(h_2,D_xG(x,y).h_1)\no\\
&\quad\qquad+\langle D^2_{x}G(x,y).(h_1,h_2),D_y \Psi(x,y)\rangle_2\Big]\mu^x(\dif y).
\end{align}
In particular,  we have
\begin{align}\label{fff}
\| D_x\bar F(x).h_1\|_1\leq C_0\|h_1\|_1,\quad
\| D_x^2\bar F(x).(h_1,h_2)\|_1\leq C_0\| h_1\|_1\| h_2\|_1,
\end{align}
where $C_0>0$ is a constant.
\el

\br
The interesting point in formula (\ref{trans1}) (and also in (\ref{trans2})) lies in that the regularity of the averaged function $\bar F$ with respect to the $x$-variable is transferred to the regularity of the solution $\Psi$ with respect to the $y$-variable. Since $\Psi$ is the solution to the corresponding Poisson equation, we can get the required regularity for free by  the uniform ellipticity property of the generator $\cL_2(x,y)$ as been proven in Theorem \ref{PP}.
\er
\begin{proof}
Recall that
$$
\cL_2(x,y)\Psi(x,y)=-\delta F(x,y)=-(F(x,y)-\bar F(x)).
$$
Note that $\delta F$ satisfies the centering condition (\ref{cen}). As in the proof of  (\ref{d1}), we have
\begin{align*}
{\mathcal{L}_2}(x,y)D_x\Psi(x,y).h_1=-D_x\delta F(x,y).h_1-\<D_xG(x,y). h_1,D_y\Psi(x,y)\>_2.
\end{align*}
Moreover, we have
$$
\int_{H_2}\Big[D_x\delta F(x,y).h_1+\langle D_xG(x,y). h_1,D_y\Psi(x,y)\rangle_2\Big]\mu^x(\dif y)=0.
$$
Note that
\begin{align*}
\int_{H_2}D_x\bar F(x).h_1\mu^x(\dif y)=D_x\bar F(x).h_1,
\end{align*}	
hence we get (\ref{trans1}). Similarly, as in (\ref{c2}) we have
\begin{align*}
&\int_{H_2}\Big[D^2_{x}\delta F(x,y).(h_1,h_2)
+2D_yD_x\Psi(x,y).(h_2,D_xG(x,y).h_1)\\
&\qquad\quad+\langle D^2_{x}G(x,y).(h_1,h_2),D_y\Psi(x,y)\rangle_2\Big]\mu^x(\dif y)=0,
\end{align*}
which in turn yields (\ref{trans2}).
Finally, due to the fact that for any $p\geq 1$,
$$
\int_{H_2}\big(1+\|y\|_2\big)^p\mu^x(\dif y)<\infty,
$$
we get estimate (\ref{fff}).
	\end{proof}

\section{Strong convergence in the averaging principle}

\subsection{Galerkin approximation}

It\^o's formula will be used frequently below in the proof of the main result. However, due to the presence of unbounded operators in the equation, we can not apply It\^o's formula for SPDE (\ref{spde1}) directly. For this reason, we introduce the following Galerkin approximation scheme.

For $n\in\mN$, let $H_1^n:= span\{e_{1,k};1\leq k\leq n\}$ and denote  the orthogonal projection of $H_1$ onto $H_1^n$ by $P_{1}^n$. Recall that $G_n(x,y)$ is defined by (\ref{gn}), and for $(x,y)\in H^n_1\times H^n_2$, define
$F_n(x,y):=P_{1}^nF(x,y)$.
We reduce the infinite dimensional system (\ref{spde1}) to the following finite dimensional system in $H_1^n\times H_2^n$:
\begin{equation}\label{xyz}
\left\{ \begin{aligned}
&\dif X^{n,\eps}_t =AX^{n,\eps}_t\dif t+F_n(X^{n,\eps}_t, Y^{n,\eps}_t)\dif t+P_{1}^n\dif W^1_t,\\
&\dif Y^{n,\eps}_t =\eps^{-1}BY^{n,\eps}_t\dif t+\eps^{-1}G_n(X^{n,\eps}_t, Y^{n,\eps}_t)\dif t+\eps^{-1/2} P_{2}^n\dif W_t^2,
\end{aligned} \right.
\end{equation}
with initial values $X_0^{n,\eps}=x^n:=P^n_1x\in H_1^n$ and $Y_0^{n,\eps}=y^n:=P^n_2y\in H_2^n$.
It is easy to check that $F_n$ and $G_n$ satisfy the same conditions as $F$ and $G$ with bounds which are uniform with respect to $n$. Thus the equation (\ref{xyz}) is well-posed in $H_1^n\times H_2^n$.
The corresponding averaged equation for system (\ref{xyz}) can be formulated as
\begin{align}\label{spden2}
\dif \bar{X}^n_t=A\bar{X}^n_t\dif t+\bar{F}_n(\bar{X}^n_t)\dif t+P_{1}^{n}\dif W_t^1,\quad \bar{X}^n_0=x^{n}\in H_1^n,
\end{align}
where $\bar{F}_n(x)$ is defined by
\begin{align}\label{Fn}
\bar F_n(x):=\int_{H_2^n}F_n(x,y)\mu^x_n(\dif y).
\end{align}
Note that  $\bar X_t^n$ is not the Galerkin approximation of $\bar X_t$.

The following result states the convergence of the finite dimensional system (\ref{xyz}) to the initial equation (\ref{spde1}).

\bl\label{xyn}
Let $T>0$,  $x\in \cD((-A)^\theta)$ and $y\in\cD((-B)^\theta)$ with $\theta\in[0,1]$. Then for every $q\geq 1$ and $\gamma\in[0,\theta],$  we have for every $t\in[0,T]$,
\begin{align}\label{axn}
\lim\limits_{n\to\infty}\Big(\mE\|(-A)^{\gamma}(X_t^\eps-X_t^{n,\eps})\|_1^q
&+\mE\|Y_t^\eps-Y_t^{n,\eps}\|_2^q\no\\
&+\mE\|(-A)^{\gamma}(\bar X_t-\bar X_t^{n})\|_1^q\Big)=0.
\end{align}
\el
\begin{proof}
When $\theta=0$ (and thus $\gamma=0$), (\ref{axn}) was proven in \cite[Lemma 4.2]{Br1}.
For general $\gamma\in(0,\theta]$,  by Lemma {\ref{la44}} we know that $X_t^\eps\in\cD((-A)^\gamma).$ As a result, we have
\begin{align*}
\mE\Big(\|(-A)^{\gamma}(X_t^\eps-X_t^{n,\eps})\|_1^q\Big)\!=\mE\!\left[\left(\sum\limits_{k=n+1}^{\infty}\alpha_k^{2\gamma}\<X_t^\eps
,e_{1,k}\>_1^2\right)^{q/2}\right]\!\!\to 0\;\;\mbox{as}\;\;n\to\infty.
\end{align*}
Furthermore, by  Lemma \ref{bx} we deduce that
\begin{align*}
\mE\Big(\|(-A)^{\gamma}(\bar X_t-\bar X_t^{n})\|_1^q\Big)&\leq \mE\left(\int_0^t\|(-A)^{\gamma}e^{(t-s)A}[\bar F(\bar X_s)-\bar F_n(\bar X_s)]\|_1\dif s\right)^q\\
&+\mE\bigg\|\int_0^t(-A)^{\gamma}e^{(t-s)A}(I-P_1^n)\dif W_s\bigg\|_1^q\\
&+\mE\!\left(\int_0^t\|(-A)^{\gamma}e^{(t-s)A}[\bar F_n(\bar X_s)-\bar F_n(\bar X^n_s)]\|_1\dif s\right)^q.
\end{align*}
Since $\|\bar F_n-\bar F\|_1\to0$ as $n\to\infty$ (see e.g. \cite[(4.4)]{Br1}), the first two terms go to 0 as $n\to\infty$ by the dominated convergence theorem. For the last term,  we have
\begin{align*}
&\lim\limits_{n\to\infty}\mE\left(\int_0^t\|(-A)^{\gamma}e^{(t-s)A}[\bar F_n(\bar X_s)-\bar F_n(X^n_s)]\|_1\dif s\right)^q\\
&\leq \lim\limits_{n\to\infty}C_1\,\mE\left(\int_0^t(t-s)^{-\gamma}\|\bar X_s-\bar X^n_s\|_1\dif s\right)^q=0,
\end{align*}
which in turn yields the desired result.
\end{proof}

\subsection{Proof of Theorem \ref{main1}}
Let $T>0$,  $x\in \cD((-A)^\theta)$ and $y\in\cD((-B)^\theta)$ with $\theta\in[0,1]$. Note that for any $t\in[0,T]$, $q\geq 1$ and $\gamma\in[0,\theta\wedge1/2)$, we have
\begin{align*}
\mE\|(-A)^\gamma (X_t^{\eps}&-\bar X_t)\|_1^q\leq \mE\|(-A)^\gamma (X_t^{\eps}- X^{n,\eps}_t)\|_1^q\no\\
&+\mE\|(-A)^\gamma (X_t^{n,\eps}-\bar X^n_t)\|_1^q+\mE\|(-A)^\gamma (\bar X^n_t-\bar X_t)\|_1^q.
\end{align*}
By Lemma \ref{xyn}, the first and the last terms on the right-hand side of the above inequality converge to $0$ as $n\to\infty$. Therefore,
in order to prove Theorem {\ref{main1}}, we only need to show that
\begin{align}\label{nxx}
\sup_{t\in[0,T]}\mE\|(-A)^\gamma (X_t^{n,\eps}-\bar X^n_t)\|_1^q\leq C_T\,\eps^{q/2},
\end{align}
where {\bf $C_T>0$ is a constant independent of $n$}.
In the following subsection, we shall only work with the approximation system (\ref{xyz}), and prove
bounds that are uniform with respect to the dimension. But in order to simplify the notations, we omit the index $n$. In particular, for $i=1,2,$ the spaces $H_i^n$ are denoted by $H_i.$

Define
\begin{align}\label{L1}
\cL_1\varphi(x,y)&:=\cL_1(x,y)\varphi(x,y):=\langle Ax+F(x,y), D_x\varphi(x,y)\rangle_1\no\\
&\quad+\frac{1}{2}Tr\big[D_x^2\varphi(x,y)Q_1\big],\quad\forall \varphi\in C_p^{2,0}(H_1\times H_2).
\end{align}
We first establish the following strong fluctuation estimate for an appropriate integral functional of $(X_r^\eps,Y_r^\eps)$ over the time interval $[s,t],$ which will play an important role in proving (\ref{nxx}).

\bl[Strong fluctuation estimate]\label{strf}
Let $T,\theta>0$, $x\in \cD((-A)^\theta)$ and $y\in\cD((-B)^\theta)$. Assume that  ({\bf A1}) and ({\bf A2})  hold, $F\in C^{2,\eta}_p(H_1\times H_2,H_1)$ and $G\in C^{2,\eta}_b(H_1\times H_2,H_2)$ with $\eta>0$. Then for any $\gamma\in[0,\theta\wedge1/2)$, $q\geq 1$, $0\leq s\leq t\leq T$ and $\tilde\phi \in C_p^{2,\eta}(H_1\times H_2, H_1)$ satisfying (\ref{cen}), we have
\begin{align}\label{stfe}
\mE\left\|\int_s^t(-A)^\gamma e^{(t-r)A}\tilde\phi(X_r^\eps,Y_r^\eps)\dif r\right\|_1^q\leq C_{q,\gamma,T}(t-s)^{(\theta-\gamma)}q\,\eps^{q/2},
\end{align}
where $C_{q,\gamma,T}>0$ is a constant.
\el

\begin{proof}
Let $\tilde\psi$  solve the Poisson equation
$$
\cL_2(x,y)\tilde\psi(x,y)=-\tilde \phi(x,y),
$$
and define
\begin{align}\label{psi8}
\tilde\psi_{t,\gamma}(r,x,y):=(-A)^\gamma e^{(t-r)A}\tilde\psi(x,y).
\end{align}
Since $\cL_2$ is an operator with respect to the $y$-variable, one can check that
\begin{align}\label{ppo}
\cL_2(x,y)\tilde\psi_{t,\gamma}(r,x,y)=-(-A)^\gamma e^{(t-r)A}\tilde\phi(x,y).
\end{align}
According to Theorem \ref{PP}, we know that $\tilde \psi\in C_p^{2,0}(H_1\times H_2,H_1)\cap\mC_p^{0,2}(H_1\times H_2,H_1).$ Applying It\^o's formula to $\tilde\psi_{t,\gamma}(t, X_t^\eps,Y_t^{\eps})$ we  get
\begin{align}\label{ito1}
\tilde\psi_{t,\gamma}(t, X_t^\eps,Y_t^{\eps})&=\tilde\psi_{t,\gamma}(s, X_s^\eps,Y_s^{\eps})+\int_s^t (\p_r+\cL_1)\tilde\psi_{t,\gamma}(r,X_r^\eps,Y_r^{\eps})\dif r\no\\
&\quad+\frac{1}{\eps}\int_s^t\cL_2\tilde\psi_{t,\gamma}(r,X_r^\eps,Y_r^{\eps})\dif r+M_{t,s}^1+\frac{1}{\sqrt{\eps}}M_{t,s}^2,
\end{align}
where $M_{t,s}^1$ and $M_{t,s}^2$ are  defined by
\begin{align*}
M_{t,s}^1:=\int_s^t D_x\tilde\psi_{t,\gamma}(r,X_r^\eps,Y_r^{\eps})\dif W_r^1\quad\text{and}\quad M_{t,s}^2:=\int_s^t D_y\tilde\psi_{t,\gamma}(r,X_r^\eps,Y_r^{\eps})\dif W_r^2.
\end{align*}
Multiplying  both sides of (\ref{ito1}) by $\eps$ and using (\ref{ppo}), we obtain
\begin{align*}
&\int_s^t(-A)^\gamma e^{(t-r)A}\tilde\phi(X_r^\eps,Y_r^\eps)\dif r\\
&=-\int_s^t\cL_2\tilde\psi_{t,\gamma}(r,X_r^\eps,Y_r^{\eps})\dif r=\eps\,\big[\tilde\psi_{t,\gamma}(s, X_s^\eps,Y_s^{\eps})-\tilde\psi_{t,\gamma}(t,X^\eps_t,Y_t^{\eps})\big]\\
&\quad+\eps\int_s^t(\p_r+\cL_1)\tilde\psi_{t,\gamma}(r,X_r^\eps,Y_r^{\eps})\dif r+\eps\, M_{t,s}^1+\sqrt{\eps}\,M_{t,s}^2.
\end{align*}
Note  that
\begin{align*}
\int_s^t\p_r\tilde\psi_{t,\gamma}(r,X_r^\eps,Y_r^{\eps})\dif r&=\int_s^t\p_r\tilde\psi_{t,\gamma}(r,X_t^\eps,Y_t^{\eps})\dif r\\
&\quad+\int_s^t\p_r\Big[\tilde\psi_{t,\gamma}(r,X_r^\eps,Y_r^{\eps})-\tilde\psi_{t,\gamma}(r,X_t^\eps,Y_t^{\eps})\Big]\dif r\\
&=\tilde\psi_{t,\gamma}(t,X^\eps_t,Y_t^{\eps})-\tilde\psi_{t,\gamma}(s,X^\eps_t,Y_t^{\eps})\\
&\quad+\int_s^t\p_r\Big[\tilde\psi_{t,\gamma}(r,X_r^\eps,Y_r^{\eps})-\tilde\psi_{t,\gamma}(r,X_t^\eps,Y_t^{\eps})\Big]\dif r,
\end{align*}
and that
$$
\p_r\tilde\psi_{t,\gamma}(r,x,y)=(-A)^{1+\gamma} e^{(t-r)A}\tilde\psi(x,y).
$$
As a result, we further get
\begin{align*}
&\int_s^t(-A)^\gamma e^{(t-r)A}\tilde\phi(X_r^\eps,Y_r^\eps)\dif r
=\eps\,(-A)^{\gamma} e^{(t-s)A}\big[\tilde\psi(X^\eps_s,Y_s^{\eps})-\tilde\psi(X^\eps_t,Y_t^{\eps})\big]\\
&+\eps\int_s^t(-A)^{1+\gamma} e^{(t-r)A}\left(\tilde\psi(X^\eps_r,Y_r^{\eps})-\tilde\psi(X^\eps_t,Y_t^{\eps})\right)\dif r\\
&+\eps\int_s^t\cL_1\tilde\psi_{t,\gamma}(r,X_r^\eps,Y_r^{\eps})\dif r+\eps\, M_{t,s}^1+\sqrt{\eps}\,M_{t,s}^2.
\end{align*}
Thus for any $0\leq s\leq t \leq T$ and $q\geq 1$, we deduce that
\begin{align*}
&\mE\left\|\int_s^t(-A)^\gamma e^{(t-r)A}\tilde\phi(X_r^\eps,Y_r^\eps)\dif r\right\|_1^q\\
&\leq C_0\,\bigg(\eps^q\,\mE\big\|(-A)^{\gamma} e^{(t-s)A}\big[\tilde\psi(X^\eps_s,Y_s^{\eps})-\tilde\psi(X^\eps_t,Y_t^{\eps})\big]\big\|_1^q\\
&\quad+\eps^q\,\mE\left\|\int_s^t (-A)^{1+\gamma} e^{(t-r)A}\left(\tilde\psi(X^\eps_r,Y_r^{\eps})-\tilde\psi(X^\eps_t,Y_t^{\eps})\right)\dif r\right\|_1^q\\
&\quad+\eps^q\,\mE\left\|\int_s^t \cL_1\tilde\psi_{t,\gamma}(r,X_r^\eps,Y_r^{\eps})\dif r\right\|_1^q+\eps^q\,\mE\|M_{t,s}^1\|_1^q+\eps^{q/2}\,\mE\|M_{t,s}^2\|_1^q\bigg)\\
&=:\sum_{i=1}^5\sJ_i(t,s,\eps).
\end{align*}
For the first term, by Lemmas {\ref{la41}}, \ref{la42},  \ref{la43} below and the fact that $\theta<1/2$,
we have
\begin{align*}
\sJ_{1}(t,s,\eps)\!&\leq\! C_1\,\eps^q\, (t-s)^{-\gamma q}\Big(\mE\big(1+\|X_t^\eps\|_1+\|X_s^\eps\|_1
+\|Y_t^\eps\|_2^p+\|Y_s^\eps\|_2^p\big)^{2q}
\Big)^{1/2}\\
&\qquad\qquad\qquad\quad\qquad\cdot\Big(\mE\| X_t^\eps-X_s^\eps\|_1^{2q}+\mE\|Y_t^{\eps}-Y_s^{\eps}\|_2^{2 q}\Big)^{1/2}\\
&\leq C_1\,(t-s)^{(\theta-\gamma)q}\eps^{(1-\theta)q}
\leq C_1\,(t-s)^{(
\theta-\gamma)q}\eps^{q/2}.
\end{align*}
Similarly, by Minkowski's inequality  we also have
\begin{align*}
\sJ_{2}(t,s,\eps)\!
&\leq C_2\,\eps^q\Bigg(\!\!\int_s^t(t-r)^{-1-\gamma}\bigg[\Big(\mE\big[\| X_t^\eps- X_r^\eps\|_1^{2q}\big]\Big)^{1/2q}\\
&\qquad\qquad\quad\qquad\qquad\qquad+\Big(\mE\big[\|Y_t^{\eps}-Y_r^{\eps}\|_2^{2 q}\big]
\Big)^{1/2q}\bigg]\dif r\Bigg)^q\\
&\leq C_2\,\eps^q\left(\int_s^t(t-r)^{-1-\gamma}\frac{(t-s)^{\theta}}{\eps^{\theta}}\dif r\right)^q\\
&\leq  C_2\,(t-s)^{(\theta-\gamma)q}\eps^{(1-\theta)q}
\leq C_2\,(t-s)^{(\theta-\gamma)q}\eps^{q/2}.
\end{align*}
To control the third term, by definitions (\ref{L1}), (\ref{psi8}) and Theorem {\ref{PP}}, one can check that
\begin{align*}
\|\cL_1\tilde\psi_{t,\gamma}(r,x,y)\|_1\leq C_3\,(t-r)^{-\gamma}\big(1+\| A x\|_1+\|y\|_2^{p})(1+\| x\|_1+\|y\|_2^{p}),
\end{align*}
which in turn yields by  Minkowski's inequality  and Lemma {\ref{la44}} that for $\gamma'\in(0,1/2)$,
\begin{align*}
\sJ_{3}(t,s,\eps)
&\leq C_3\,\eps^{(1-\gamma') q}\left(\int_s^t (t-r)^{-\gamma}r^{(\theta-1)}\dif r\right)^q\\
&\leq C_3\,(t-s)^{(\theta-\gamma)q}\eps^{(1-\gamma')q}\leq C_3\,(t-s)^{(\theta-\gamma)q}\eps^{q/2}.
\end{align*}
As for $\sI_4(t,s,\eps)$,  by  Burkholder-Davis-Gundy's inequality and the assumption ({\bf A2}), we have
\begin{align*}
&\sJ_{4}(t,s,\eps)\leq C_4\,\eps^{q}\left(\int_s^t\mE\|(-A)^\gamma e^{(t-r)A}D_x\tilde\psi(X_r^\eps,Y_r^{\eps})Q_1^{1/2}\|_{\sL_2(H_1)}^{2}\dif  r\right)^{q/2}\\
&\leq  C_4\,\eps^{q}\left(\int_s^t\Big(1+\mE\| X_r^\eps\|_1^{2}+\mE\|Y_r^{\eps}\|_2^{2p}\Big)\|(-A)^\gamma  e^{(t-r)A}Q_1^{1/2}\|_{\sL_2(H_1)}^2\dif  r\right)^{q/2}\\
&\leq  C_4\,(t-s)^{(1/2-\gamma)q}\eps^{q}\leq  C_4\,(t-s)^{(\theta-\gamma)q}\eps^{q},
\end{align*}
and similarly one can check that
\begin{align*}
\sJ_{5}(t,s,\eps) \leq  C_5\,(t-s)^{(1/2-\gamma)q}\eps^{q/2}\leq  C_5\,(t-s)^{(\theta-\gamma)q}\eps^{q/2}.
\end{align*}
Combining the above computations, we get
the desired estimate.
\end{proof}

Now, we are in the position to give:

\begin{proof}[Proof of estimate (\ref{nxx})]
	
Fix $T>0$ below. In view of (\ref{spde22}) and (\ref{msb}), we have for every $t\in[0,T]$ and $\gamma\in[0,\theta\wedge1/2)$,
\begin{align*}
(-A)^\gamma (X_t^\eps-\bar X_t)&=\int_0^t(-A)^\gamma e^{(t-s)A}\big[\bar F(X_s^\eps)-\bar F(\bar X_s)\big]\dif s\\
&\quad+\int_0^t(-A)^\gamma e^{(t-s)A}\delta F(X_s^\eps,Y_s^{\eps})\dif s,
\end{align*}
where $\delta F$ is defined by (\ref{dF}).
Thus for any   $q\geq 1$, we have
\begin{align*}
\mE \|(-A)^\gamma (X_t^\eps-\bar X_t)\|_1^q&\leq C_q\, \mE\left\|\int_0^t(-A)^\gamma e^{(t-s)A}\big[\bar F(X_s^\eps)-\bar F(\bar X_s)\big]\dif s\right\|_1^q\\
&\quad+C_q\,\mE\left\|\int_0^t(-A)^\gamma e^{(t-s)A}\delta F(X_s^\eps,Y_s^{\eps})\dif s\right\|_1^q\\
&=:\sI_1(t,\eps)+\sI_2(t,\eps).
\end{align*}
By Lemma  \ref{aveF} and Minkowski's inequality, we deduce that
\begin{align*}
\sI_1(t,\eps)
&\leq C_1\,\mE\left(\int_0^{t}(t-s)^{-\gamma}\|\bar F(X_s^\eps)-\bar F(\bar X_s)\|_1\dif s\right)^q\\
&\leq C_1\,\left(\int_0^{t}(t-s)^{-\gamma} \Big(\mE\|X_s^\eps-\bar X_s\|_1^q\Big)^{1/q}\dif s\right)^q.
\end{align*}
For the second term, note that $\delta F(x,y)$ satisfies the centering condition (\ref{cen}). As a result, it follows by Lemma \ref{strf} directly that
\begin{align*}
\sI_{2}(t,\eps)\leq C_2 \,\eps^{q/2}.
\end{align*}
Thus we arrive at
\begin{align}\label{77}
\mE \|(-A)^\gamma (X_t^\eps-\bar X_t)\|_1^q&\leq C_3 \,\eps^{q/2}\no\\
&+C_3\,\left(\int_0^t(t-s)^{-\gamma}\Big(\mE\|X_s^\eps-\bar X_s\|_1^q\Big)^{1/q}\dif s\right)^q.
\end{align}
Letting $\gamma=0$, by  Gronwall's inequality we obtain
\begin{align*}
\sup_{t\in[0,T]}\mE \|X_t^\eps-\bar X_t\|_1^q\leq C_4\,\eps^{q/2}.
\end{align*}
Taking this back into (\ref{77}), we get the desired result.
\end{proof}

\section{Normal deviations}

\subsection{Kolmogorov equation}

Recall that $\bar X_t$ and $\bar Z_t$ satisfy the equations (\ref{spde2}) and (\ref{spdez}), respectively. We write a system of equations for the process $(\bar X_t, \bar Z_t)$ as follows:
\begin{equation*}
\left\{ \begin{aligned}
&\dif \bar{X}_t=A\bar{X}_t\dif t+\bar{F}(\bar{X}_t)\dif t+\dif W_t^1,\qquad\qquad\qquad\, \bar X_0=x,\\
&\dif \bar Z_t=A\bar Z_t\dif t+D_x\bar F(\bar X_t).\bar Z_t\dif t+\sigma(\bar X_t)\dif \tilde W_t,\qquad\bar Z_0=0.
\end{aligned} \right.
\end{equation*}
Note that the processes $\bar X_t$ and $\bar Z_{t}$ depend on the initial value $x$. Below, we shall write $\bar X_t(x)$  when we want to stress its dependence on the
initial value, and use $\bar Z_{t}(x,z)$ to denote the process $\bar Z_{t}$ with initial point $\bar Z_0=z\in H_1$.
Let $\bar \cL$ be the formal infinitesimal generator of the  Markov process $(\bar X_t, \bar Z_t)$, i.e.,
$$
\bar \cL:=\bar \cL_1+\bar \cL_3,
 $$
 where for every $\varphi\in C_p^2(H_1)$, $\bar \cL_1$ and $\bar \cL_3$ are defined by
\begin{align}
\!\!\!\!\!\!\!\bar\cL_1\varphi(x):=\bar\cL_1(x)\varphi(x)&:=\langle Ax+\bar F(x), D_x\varphi(x)\rangle_1+\frac{1}{2}Tr[D_x^2\varphi(x)Q_1],\label{bL1}\\
\bar \cL_3\varphi(z):=\bar \cL_3(x,z)\varphi(z)&:=\<Az+D_x\bar F(x)z,D_z\varphi(z)\>_1\no\\
&\;\quad\,+\frac{1}{2}\,Tr\big[D^2_{z}\varphi(z)\sigma(x)\sigma^*(x)\big].\label{bL3}
\end{align}

  Fix $T>0$, consider the following Cauchy problem on $[0,T]\times H_1\times H_1$:
\begin{equation} \label{kez}
\left\{ \begin{aligned}
&\partial_t\bar u(t,x,z)=\bar \cL\,\bar u(t,x,z),\quad t\in(0,T],\\
& \bar u(0,x,z)=\varphi(z),\\
\end{aligned} \right.
\end{equation}
where $\varphi: H_1\to\mR$ is measurable.
We have the following result, which will be used below to prove the weak convergence of $Z_t^\eps$ to $\bar Z_t$.

\bt\label{lako}
For every $\varphi\in \mC_b^4(H_1)$, there exists a   solution $\bar u\in C_b^{1,2,4}([0,T]\times H_1\times H_1)$  to the equation (\ref{kez}) which is given by
\begin{align}\label{prou}
\bar u(t,x,z)=\mE\big[\varphi(\bar Z_t(x,z))\big].
\end{align}
Moreover, we have:
	
	\vspace{1mm}
	\noindent
(i) For any  $t\in(0,T]$, $x,z\in H_1$ and $h\in \cD((-A)^{\beta})$ with $\beta\in[0,1]$,
	\begin{align}\label{uz}
	|D_z\bar u(t,x,z).(-A)^\beta h|\leq C_1\,t^{-\beta}\| h\|_1;
	\end{align}
(ii) For any $t\in(0,T]$, $x,z\in H_1$, $h_1\in \cD((-A)^{\beta_1})$ and $h_2\in \cD((-A)^{\beta_2})$  with $\beta_1,\beta_2\in[0,1]$,
	\begin{align}\label{uz2}
	| D^2_{z}\bar u(t,x,z).((-A)^{\beta_1}h_1,(-A)^{\beta_2}h_2)|\leq C_2\,t^{-\beta_1-\beta_2}\| h_1\|_1\| h_2\|_1,
	\end{align}
 and for any $x,z,h_2\in H_1$ and $h_1\in \cD((-A)^{\beta})$ with $\beta\in[0,1]$,
\begin{align}\label{uzx}
	| D_xD_{z}\bar u(t,x,z).((-A)^{\beta}h_1,h_2)|\leq C_2\,t^{-\beta}\| h_1\|_1\| h_2\|_1;
	\end{align}
(iii) For any $t\in(0,T]$, $x,z\in H_1$, $h_1\in \cD((-A)^{\beta_1}),h_2\in \cD((-A)^{\beta_2})$ and $h_3\in \cD((-A)^{\beta_3})$ with $\beta_1,\beta_2,\beta_3\in[0,1]$,
	\begin{align}\label{uz3}
	&| D_{z}^3\bar u(t,x,z).((-A)^{\beta_1}h_1,(-A)^{\beta_2}h_2,(-A)^{\beta_3}h_3)|\no\\
&\leq C_3\,t^{-\beta_1-\beta_2-\beta_3}\| h_1\|_1\| h_2\|_1\|h_3\|_1,
\end{align}
and for any $x,z,h_3\in H_1$, $h_1\in \cD((-A)^{\beta_1})$ and $h_2\in \cD((-A)^{\beta_2})$  with $\beta_1,\beta_2\in[0,1]$,
\begin{align}\label{uz2x}
&| D_xD_{z}^2\bar u(t,x,z).((-A)^{\beta_1}h_1,(-A)^{\beta_2}h_2,h_3)|\no\\
&\leq C_3\,t^{-\beta_1-\beta_2}\| h_1\|_1\| h_2\|_1\|h_3\|_1;
\end{align}
(iv) For any $t\in(0,T]$, $x\in \cD(-A)$ and $z, h\in H_1$,
	\begin{align}\label{utz}
	\vert \p_tD_z\bar u(t,x,z).h|\leq C_4\,\Big(t^{-1}(1+\| z\|_{1})+\|Ax\|_1+\|x\|_1^2\Big)\|h\|_1;
	\end{align}
\noindent(v) For any  $t\in(0,T]$, $x\in \cD(-A)$ and $z, h_1,h_2\in H_1$,
	\begin{align}\label{utzz}
	\vert \p_tD_z^2\bar u(t,x,z).(h_1,h_2)|\!\!
\leq\! C_5\!\Big(t^{-1}(1\!+\!\| z\|_{1})\!+\!\!\|Ax\|_1\!+\!\|x\|_1^2\Big)\!\|h_1\|_1\|h_2\|_1,
	\end{align}
and for any $x\in \cD(-A),z,h_1\in H_1$ and $h_2\in \cD((-A))$,
\begin{align}\label{utxz}
\vert \p_tD_xD_z\bar u(t,x,z).(h_1,h_2)|&\leq C_5\,\Big(t^{-1}(1+\|z\|_1)+\|Ax\|_1+\|x\|_1^2\Big)\|h_1\|_1\|h_2\|_{1}\no\\
&\quad+C_5\,\|h_1\|_1\|Ah_2\|_{1};
\end{align}
where $C_i$, $i=1,\cdots, 5,$ are positive constants.
\et

\br
The estimates in {\it (i)}-{\it (iii)} have been studied in  \cite[Proposition 7.1]{Br2} when the diffusion coefficient is a constant and in \cite[Theorem 4.2, Theorem 4.3 and Proposition 4.5]{Br4} for general nonlinear diffusion coefficients. However, the index $\beta$ in (\ref{uz}), (\ref{uzx}) and (\ref{uz2x}), $\beta_1, \beta_2, \beta_3$ in (\ref{uz2}) and (\ref{uz3}) are restricted to $[0,1)$, which is not sufficient for us to use below. The key observation here is that the equation (\ref{spdez}) satisfied by $\bar Z_t$ is a linear one, and we do not involve estimates for $D_x\bar u$ and $D^2_x\bar u$.
Thus some new techniques are needed in the proof of Theorem \ref{main3} to avoid using these estimates.
\er

\begin{proof}
{\it (i)}-{\it (iii).} By using the same argument as in  \cite[Theorem 13]{Da},  we can prove that $\bar u$ defined by (\ref{prou}) is a solution to the equation (\ref{kez}). Moreover, $\bar u $ has bounded Gat\^eaux derivatives with respect to the $x$-variable up to order $2$ and with respect to the $z$-variable up to order $4$, see also \cite[Section 7]{Br2} and \cite[Section 4]{Br4}. Furthermore, in view of (\ref{prou})  we deduce that for any $\beta\in[0,1]$,
$$
D_z\bar u(t,x,z).(-A)^\beta h=\mE\Big[\<\varphi'(\bar Z_t(x,z)),D_z\bar Z_t(x,z).(-A)^\beta h\>_1\Big].
$$
Since $\bar Z_t$ satisfies (\ref{spdez}), we thus have
$$
\dif (D_z\bar Z_t(x,z).(-A)^\beta h)=\big(A+D_x\bar F(\bar X_t)\big).(D_z\bar Z_t(x,z).(-A)^\beta h)\dif t,
$$
and the initial value is given by $D_z\bar Z_0(x,z).(-A)^\beta h=(-A)^\beta h.$
As a result,
$$
\|D_z\bar Z_t(x,z).(-A)^\beta h\|_1\leq C_0\,t^{-\beta}\|h\|_1,
$$
which in turn yields (\ref{uz}).  Estimates (\ref{uz2})-(\ref{uz2x}) can be proved similarly, hence we omit the details here.

\vspace{1mm}
\noindent {\it (iv)}
To  prove estimate (\ref{utz}), by (\ref{kez}) we note that for any $h\in H_1,$
\begin{align}\label{ptdz}
\partial_tD_z\bar u(t,x,z).h=D_z\partial_t\bar u(t,x,z).h=D_z(\bar\cL_1+\bar\cL_3)\bar u(t,x,z).h.
\end{align}
By definition (\ref{bL1}) we have
\begin{align}\label{ptdz2}
D_z\bar\cL_1\bar u(t,x,z).h&=D_zD_x\bar u(t,x,z).(Ax+\bar F(x),h)\no\\
&\quad+\frac{1}{2}\sum\limits_{n=1}^\infty \lambda_{1,n}D_zD_x^2\bar u(t,x,z).(e_{1,n},e_{1,n},h),
\end{align}
which implies that
\begin{align}\label{lxdz}
&|D_z\bar\cL_1\bar u(t,x,z).h|\leq C_1(1+\|Ax\|_1)\|h\|_1.
\end{align}
Similarly, by definition  (\ref{bL3}) we have
\begin{align}\label{ptdz3}
D_z\bar\cL_3\bar u(t,x,z).h&=\<Ah+D_x\bar F(x).h,D_z\bar u(t,x,z)\>_1\no\\
&\quad+D_z^2\bar u(t,x,z).(A z+D_x\bar F(x).z,h)\no\\
&\quad+\frac{1}{2}\sum\limits_{n=1}^{\infty} D_z^3\bar u(t,x,z).(\sigma(x)e_{1,n},\sigma(x)e_{1,n},h),
\end{align}
which together with (\ref{uz}) and (\ref{uz2}) yields that
\begin{align}\label{lzdz}
&|D_z\bar\cL_3\bar u(t,x,z).h|\leq C_2\,t^{-1}(1+\|z\|_1)\|h\|_1+(1+\|x\|_1^2)\|h\|_1.
\end{align}
Combining (\ref{ptdz}), (\ref{lxdz}) and (\ref{lzdz}), we obtain (\ref{utz}).

\vspace{1mm}
\noindent {\it (v)} In view of (\ref{ptdz}), (\ref{ptdz2}) and (\ref{ptdz3}), we note that for any $h_1,h_2\in H_1,$
\begin{align*}
\partial_tD_z^2\bar u(t,x,z).(h_1,h_2)&=D_z^2D_x\bar u(t,x,z).(Ax+\bar F(x),h_1,h_2)\\
&\quad+D_z^2\bar u(t,x,z).(Ah_1+D_x\bar F(x).h_1,h_2)\\
&\quad+D_z^2\bar u(t,x,z).(Ah_2+D_x\bar F(x).h_2,h_1)\\
&\quad+D_z^3\bar u(t,x,z).(A z+D_x\bar F(x).z,h_1,h_2)\\
&\quad+\frac{1}{2}\sum\limits_{n=1}^\infty \lambda_{1,n}D_z^2D_x^2\bar u(t,x,z).(e_{1,n},e_{1,n},h_1,h_2)\\
&\quad+\frac{1}{2}\sum\limits_{n=1}^{\infty} D_z^4\bar u(t,x,z).(\sigma(x)e_{1,n},\sigma(x)e_{1,n},h_1,h_2).
\end{align*}
Using (\ref{uz2}), (\ref{uz3}) and Lemma \ref{aveF}, one can check that
\begin{align*}
|\partial_tD_z^2\bar u(t,x,z).(h_1,h_2)|\leq C_3\,(t^{-1}+\|Ax\|_1+\|x\|_1^2+t^{-1}\|z\|_1)\|h_1\|_1\|h_2\|_1,
\end{align*}
which means that (\ref{utzz}) holds. Finally, we have
\begin{align*}
&\partial_tD_xD_z\bar u(t,x,z).(h_1,h_2)=D_xD_zD_x\bar u(t,x,z).(Ax+\bar F(x),h_1,h_2)\\
&\qquad+D_zD_x\bar u(t,x,z).(Ah_2+D_x\bar F(x).h_2,h_1)\\
&\qquad+D_xD_z\bar u(t,x,z).(Ah_1+D_x\bar F(x).h_1,h_2)\\
&\qquad+\<D_x^2\bar F(x).(h_1,h_2),D_z\bar u(t,x,z)\>_1+D_z^2\bar u(t,x,z).\big(D_x^2\bar F(x).(z,h_2),h_1\big)\\
&\qquad+D_xD_z^2\bar u(t,x,z).(A z+D_x\bar F(x).z,h_1,h_2)\\
&\qquad+\frac{1}{2}\sum\limits_{n=1}^\infty \lambda_{1,n}D_xD_zD_x^2\bar u(t,x,z).(e_{1,n},e_{1,n},h_1,h_2)\\
&\qquad+\frac{1}{2}\sum\limits_{n=1}^{\infty} D_xD_z^3\bar u(t,x,z).(\sigma(x)e_{1,n},\sigma(x)e_{1,n},h_1,h_2)\\
&\qquad+\sum\limits_{n=1}^{\infty} D_z^3\bar u(t,x,z).((D_x\sigma(x).h_2)e_{1,n},\sigma(x)e_{1,n},h_1).
\end{align*}
Using   (\ref{uzx}), (\ref{uz2x}) and Lemma \ref{aveF} we  obtain
\begin{align*}
\vert \p_tD_xD_z\bar u(t,x,z).(h_1,h_2)|&\leq C_4\Big(t^{-1}+\|Ax\|_1+\|x\|_1^2+
t^{-1}\|z\|_1\Big)\|h_1\|_1\|h_2\|_{1}\\
&\quad+C_4\|h_1\|_1\|Ah_2\|_{1},
\end{align*}
which yields (\ref{utxz}).
\end{proof}

\subsection{Estimates for $Z_t^\eps$}

Recall that $Z_t^\eps$ satisfies (\ref{zte}).
 In particular, we  have
\begin{align}\label{z7}
Z_t^{\eps}=\!\frac{1}{\sqrt{\eps}}\!\int_0^te^{(t-s)A}[\bar F(X_s^{\eps})\!-\!\bar F( \bar X_s)]\dif s+\!\frac{1}{\sqrt{\eps}}\int_0^te^{(t-s)A}\delta F(X_s^{\eps},Y_s^{\eps})\dif s.
\end{align}
Furthermore, by Theorem \ref{main1} we get that for any $q\geq 1$, $x\in \cD((-A)^\theta)$, $y\in \cD((-B)^\theta)$ with $\theta>0$  and $\gamma\in[0,\theta\wedge1/2)$,
\begin{align}\label{az}
\mE\|(-A)^{\gamma}Z_t^\eps\|_1^q<\infty.
\end{align}
We shall need  the following regularity property of $Z_t^\eps$ with respect to the time variable.

\begin{lemma}\label{rpz}
Let $T>0$, $x\in \cD((-A)^\theta)$ and $y\in \cD((-B)^\theta)$ with $\theta\in (0,1] $. Assume that ({\bf A1}) and ({\bf A2})  hold, $F\in C^{2,\eta}_p(H_1\times H_2,H_1)$ and $G\in C^{2,\eta}_b(H_1\times H_2,H_2)$ with $\eta>0$. Then  for any $q\geq 1$, $0\leq s\leq t\leq T$	and $\vartheta\in(0,\theta)$, there exists a constant $C_{q,T}>0$ such that
	\begin{align*}
\mE\|  Z_t^\eps- Z_s^\eps\|_1^q\leq C_{q,T}(t-s)^{q\vartheta}.
	\end{align*}
\end{lemma}
\begin{proof}
	By (\ref{z7}), we have
	\begin{align*}
		Z_t^\eps-Z_s^\eps&=\frac{1}{\sqrt{\eps}}\int_s^te^{(t-r)A}(\bar F(X_r^\eps)-\bar F(\bar X_r)) \dif r\\
		&\quad+\big(e^{(t-s)A}-I\big)\frac{1}{\sqrt{\eps}}\int_0^se^{(s-r)A}(\bar F(X_r^\eps)-\bar F(\bar X_r)) \dif r\\
		&\quad+\frac{1}{\sqrt{\eps}}\int_s^te^{(t-r)A}\delta F(X_r^\eps,Y_r^\eps) \dif r\\
		&\quad+\big(e^{(t-s)A}-I\big)\frac{1}{\sqrt{\eps}}\int_0^se^{(s-r)A}\delta F(X_r^\eps,Y_r^\eps) \dif r=:\sum\limits_{i=1}^4\sZ_i(t,s).
	\end{align*}
Using Minkowski's inequality and Theorem {\ref{main1}} with $\gamma=0$, we get that
	\begin{align*}
		\mE\|\sZ_1(t,s)\|_1^q\leq C_1\,\left(\frac{1}{\sqrt{\eps}}\int_s^t\big(\mE\|X_r^\eps-\bar X_r\|_1^q\big)^{1/q} \dif r\right)^q\leq C_1\,(t-s)^q.
	\end{align*}
Furthermore, by Proposition \ref{arp} (ii) we have that for any $\vartheta\in (0,1),$
\begin{align*}
	&\mE\|\sZ_2(t,s)\|_1^q\\
&\leq C_2\,(t-s)^{q\vartheta} \left(\frac{1}{\sqrt{\eps}}\int_0^s\Big(\mE\|(-A)^\vartheta e^{(s-r)A}(\bar F(X_r^\eps)-\bar F(\bar X_r))\|_1^q\Big)^{1/q} \dif r\right)^q\\&
		\leq C_2\,(t-s)^{q\vartheta} \left(\frac{1}{\sqrt{\eps}}\int_0^s(s-r)^{-\vartheta}\big(\mE\|X_r^\eps-\bar X_r\|_1^q\big)^{1/q} \dif r\right)^q\leq C_2\,(t-s)^{q\vartheta}.
	\end{align*}
Note that $\delta F(x,y)$ satisfies the centering condition (\ref{cen}). As a direct consequence of the fluctuation estimate (\ref{stfe}), we obtain that
	\begin{align*}
		\mE\|\sZ_3(t,s)\|_1^q \leq C_3\,(t-s)^{q\theta}.
	\end{align*}
Finally, by making use of Proposition \ref{arp} (ii) and (\ref{stfe}) again, we have for any $\vartheta\in(0,\theta)$,
\begin{align*}
		\mE\|\sZ_4(t,s)\|_1^q  &\leq C_4\,(t-s)^{q\vartheta}\, \mE\left\|\frac{1}{\sqrt{\eps}}\int_0^s(-A)^\vartheta e^{(s-r)A}\delta F(X_r^\eps,Y_r^\eps)\dif r\right\|_1^q\\
&\leq C_4\,(t-s)^{q\vartheta}.
	\end{align*}
	Combining the above computations, we get the desired result.
\end{proof}

As in Section 4, to prove Theorem \ref{main3} we also need to reduce the infinite dimensional problem to a finite dimensional one by the Galerkin approximation. Recall that $X_t^{n,\eps}$ and $\bar X_t^{n}$ are defined by (\ref{xyz}) and (\ref{spden2}), respectively.
Define
$$
Z_t^{n,\eps}:=\frac{X_t^{n,\eps}-\bar X_t^{n}}{\sqrt{\eps}}.
$$
Then we have
\begin{align*}
&\dif Z_t^{n,\eps}=AZ_t^{n,\eps} \dif t+\eps^{-1/2}[\bar F_n(X_t^{n,\eps})-\bar F_n( \bar X_t^n)]\dif t+\eps^{-1/2}\delta F_n(X_t^{n,\eps},Y_t^{n,\eps})\dif t,
\end{align*}
where $\bar F_n$ is given by (\ref{Fn}), and $\delta F_n(x,y):=F_n(x,y)-\bar F_n(x)$.
Let $ \bar Z_t^n$ satisfy the following linear equation:
\begin{align}\label{spdezn}
\dif \bar Z_t^n=A\bar Z_t^n\dif t+D_x\bar F_n(\bar X_t^n).\bar Z_t^n\dif t+P_{1}^n\sigma(\bar X_t^n)\dif \tilde W_t,
\end{align}
where $\tilde W_t$ is a cylindrical Wiener process in $H_1$, and $\sigma(x)$ is defined by (\ref{sst}).
We have the following approximation result.

\bl
For every $\eps>0$ and  $x,z\in H_1$, we have
\begin{align}\label{znz}
\lim\limits_{n\to\infty}\mE\Big(\|Z^\eps_t-Z_t^{n,\eps}\|_1+\|\bar Z_t-\bar Z_t^{n}\|_1\Big)=0.
\end{align}
\el

\begin{proof}
	By the definition of $Z_t^\eps,Z_t^{\eps,n}$ and Lemma {\ref{xyn}}, we have
\begin{align*}
\lim\limits_{n\to\infty}\mE\|Z^\eps_t-Z_t^{n,\eps}\|_1
\leq \lim\limits_{n\to\infty}\frac{1}{\sqrt{\eps}}\Big(\mE\|X^\eps_t-X_t^{n,\eps}\|_1
+\mE\|\bar X_t-\bar X_t^{n}\|_1\Big)=0.
\end{align*}
Furthermore, in view of (\ref{spdez}) and (\ref{spdezn}), we have
\begin{align*}
\bar Z_t-\bar Z_t^{n}&=\int_0^te^{(t-s)A}[D_x\bar F(\bar X_s).\bar Z_s-D_x\bar F(\bar X_s^{n}).\bar Z_s]\dif s\\
&\quad+\int_0^te^{(t-s)A}[D_x\bar F(\bar X_s^n).\bar Z_s-D_x\bar F_n(\bar X_s^{n}).\bar Z_s^n]\dif s\\
&\quad+\int_0^te^{(t-s)A}[\sigma(\bar X_s)-\sigma(\bar X_s^n)]P_1^n\dif \tilde W_s.
\end{align*}
Thus we deduce that
\begin{align*}
&\mE\|\bar Z_t-\bar Z_t^{n}\|_1^2\leq C_1\left(\int_0^t\left(\mE\|\bar X_s-\bar X_s^{n}\|_1^4\right)^{1/4}\left(\mE\|\bar Z_s\|_1^4\right)^{1/4}\dif s\right)^{2}\\
&+C_1\!\int_0^t\!\big(\mE\|\bar X_s\!-\!\bar X_s^{n}\|_1^2+\|D_x\bar F\!-\!D_x\bar F_n\|_{\sL(H_1)}^2\big)\dif s+C_1\!\int_0^t\mE\|\bar Z_s\!-\!\bar Z_s^n\|_1^2\dif s.
\end{align*}
By Gronwall's inequality we obtain
\begin{align*}
\mE\|\bar Z_t-\bar Z_t^{n}\|_1^2\leq C_2\,e^{C_2\, t}\left(\int_0^t\big(\mE\|\bar X_s-\bar X_s^{n}\|_1^2+\|D_x\bar F-D_x\bar F_n\|_{\sL(H_1)}^2\big)\dif s\right),
\end{align*}
which yields the desired result.
	\end{proof}

\subsection{Proof of Theorem \ref{main3}}

For any $T>0$ and $\varphi\in \mC_b^4(H_1),$ we have for $t\in[0,T]$,
\begin{align}\label{wcz}
&\left|\mE[\varphi(Z_t^{\eps})]-\mE[\varphi(\bar Z_t)]\right|\leq \left|\mE[\varphi(Z_t^{\eps})]-\mE[\varphi(Z_t^{n,\eps})]\right|\no\\
&+\left|\mE[\varphi(Z_t^{n,\eps})]-\mE[\varphi(\bar Z_t^n)]\right|+\left|\mE[\varphi(\bar Z_t^n)]-\mE[\varphi(\bar Z_t)]\right|.
\end{align}
 By making use of (\ref{znz}), the first and the last terms on the right-hand of (\ref{wcz}) converge to $0$ as $n\to\infty$. Therefore,
in order to prove Theorem {\ref{main3}}, we only need to show that for any $\gamma\in(0,1/2)$,
\begin{align}\label{nzz}
\sup_{t\in[0,T]}\left|\mE[\varphi(Z_t^{n,\eps})]-\mE[\varphi(\bar Z_t^n)]\right|\leq C_T\,\eps^{\frac{1}{2}-\gamma},
\end{align}
where {\bf $C_T>0$ is a constant independent of} $n.$ As before, we shall only work with the approximation system in the following subsection, and proceed to prove
bounds that are uniform with respect to the dimension. To simplify the notations, we shall omit the index $n$  as before.

Fix $T>0$, and for every $\varphi\in C_p^1(H_1)$, let
\begin{align}\label{333}
\cL_3^\eps\varphi(z)&:=\cL_3^\eps(x,y,\bar x,z)\varphi(z):=\<Az,D_z\varphi(z)\>_1\no\\
&+\frac{1}{\sqrt{\eps}}\<\bar F(x)-\bar F(\bar x), D_z\varphi(z)\>_1+\frac{1}{\sqrt{\eps}}\<\delta F(x,y),D_z\varphi(z)\>_1.
\end{align}
 We call a function $\phi(t,x,y,\bar x,z)$ defined on $[0,T]\times H_1\times H_2\times H_1\times H_1$ admissible,  if it is centered, i.e.,
\begin{align}\label{cen222}
\int_{H_2}\phi(t,x,y,\bar x,z)\mu^x(\dif y)=0,\quad\forall t>0, x,\bar x,z\in H_1,
\end{align}
and the following conditions hold:

\vspace{1mm}
\noindent {\bf (H):} for any $t\in[0,T)$, $x,z\in H_1$, $y\in H_2$, $\bar x\in\cD(-A)$ and $h_1,h_2\in H_1$,
\begin{align}\label{as1}
	&|\p_t\phi(t,x,y,\bar x,z)|+|D_x\p_t\phi(t,x,y,\bar x,z).h_1|+|D_z\p_t\phi(t,x,y,\bar x,z).h_2|\no\\
&\leq C_0\,(T-t)^{-1}(1+\|A\bar x\|_1+\|\bar x\|_1^2+\|z\|_1)\no\\
&\qquad\qquad\qquad\qquad\times(1+\|x\|_1+\|y\|_2^p)(\|h_1\|_1+\|h_2\|_1),
\end{align}
and for any $h_3\in\cD((-A))$,
\begin{align}\label{as0}
|D_{\bar x}\p_t\phi(t,x,y,\bar x,z).h_3|&\leq C_0\Big((T-t)^{-1}(1+\|A\bar x\|_1+\|\bar x\|_1^2+
\|z\|_1)\|h_3\|_{1}\no\\
&\qquad\quad+\|Ah_3\|_{1}\Big)\times(1+\|x\|_1+\|y\|_2^p),
\end{align}
and for any $h\in \cD((-A)^{\vartheta})$ with $\vartheta\in[0,1],$
\begin{align}\label{as4}
|D_z\phi(t,x,y,\bar x,z).(-A)^\vartheta h|\leq C_0\,(T-t)^{-\vartheta}(1+\|x\|_1+\|y\|_2^{p})\|h\|_1.
\end{align}

Given an admissible function $\phi(t,x,y,\bar x,z) \in C_p^{1,2,\eta,2,2}([0,T]\times H_1\times H_2\times H_1\times H_1)$ with $\eta>0$, let $\psi(t,x,y,\bar x,z)$ solve the following Poisson equation:
\begin{align}\label{psi1}
\cL_2(x,y)\psi(t,x,y,\bar x,z)=-\phi(t,x,y,\bar x,z),
\end{align}
and define
$$
\overline{\delta F\cdot\nabla_z\psi}(t,x,\bar x,z):=\int_{H_2}\nabla_z\psi(t,x,y,\bar x,z).\delta F(x,y)\mu^x(\dif y).
$$
The following weak fluctuation estimates for an integral functional of process $(X_t^\eps,Y_t^{\eps},\bar X_t,Z_t^{\eps})$  will play an important role in proving (\ref{nzz}). Compared with Lemma \ref{strf}, extra efforts are needed to control the time singularity in the integral.

\bl[Weak fluctuation estimates]\label{weaf}
Let $T,\theta>0,$ $x\in \cD((-A)^\theta)$ and $y\in \cD((-B)^\theta)$. Assume that ({\bf A1}) and ({\bf A2}) hold, $F\in C^{2,\eta}_p(H_1\times H_2,H_1)$ and $G\in C^{2,\eta}_b(H_1\times H_2,H_2)$.
Then for every admissible function $\phi\in C_p^{1,2,\eta,2,2}([0,T] \times H_1\times H_2\times H_1\times H_1)$ with $\eta>0$, $t\in[0,T]$ and  $\gamma\in(0,1/2),$ we have
\begin{align}
\bigg|\mE\bigg(\int_0^t\phi(t,X_s^{\eps},Y_s^{\eps},\bar X_s,Z_s^{\eps})\dif s\bigg)\bigg|&\leq C_T\,\eps^{\frac{1}{2}},  \label{we1}
\end{align}
and
\begin{align}
\bigg|\mE\bigg(\frac{1}{\sqrt{\eps}}\int_0^t\!&\phi(s,X_s^{\eps},Y_s^{\eps},\bar X_s,Z_s^{\eps})\dif s\no\\
&-\int_0^t\overline{\delta F\cdot\nabla_z\psi}(s,X_s^{\eps},\bar X_s,Z_s^{\eps})\dif s\bigg)\bigg|\!\leq C_T\,\eps^{\frac{1}{2}-\gamma},\label{we2}
\end{align}
where $C_T>0$ is a constant.
\el

\begin{proof}
We divide the proof into two steps.

\vspace{1mm}
\noindent{\bf Step 1.} We first prove estimate (\ref{we1}). By Theorem \ref{PP}, we have that $\psi\in C_p^{1,2,2,2,2}([0,T]\times H_1\times H_2\times H_1\times H_1)$. Thus we can apply It\^o's formula to $\psi(t,X_t^\eps,Y_t^\eps,\bar X_t,Z_t^\eps)$ to derive that
	\begin{align*}
	&\mE[\psi(t,X_t^{\eps},Y_t^\eps,\bar X_t,Z_t^\eps)]\\&=\psi(0,x,y,x,0)+
	\mE\left(\int_0^t(\partial_s+\cL_1+\mathcal{\bar L}_1+\cL_3^\eps)\psi(s,X_s^{\eps},Y_s^\eps,\bar X_s,Z_s^\eps)\dif s\right)\\
	&\quad+\frac{1}{\eps}\mE\left(\int_0^t\mathcal{L}_2\psi(s,X_s^{\eps},Y_s^\eps,\bar X_s,Z_s^\eps)\dif s\right),
	\end{align*}
where $\cL_1, \cL_2,\mathcal{\bar L}_1$ and $\cL_3^\eps$ are defined by (\ref{L1}), (\ref{L2}), (\ref{bL1}) and (\ref{333}), respectively.
	Multiplying  both sides of the above equality by $\eps$ and taking into account (\ref{psi1}), we obtain
	\begin{align*}
	&\bigg|\mE\bigg(\int_0^t\phi(s,X_s^{\eps},Y_s^{\eps},\bar X_s,Z_s^{\eps})\dif s\bigg)\bigg|\\
	&=\bigg|\eps\,\mE\big[\psi(0,x,y,x,0)-\psi(t,X_t^{\eps},Y_t^\eps,\bar X_t,Z_t^\eps)\big]\\
	&\quad+\eps\;\mE\bigg(\int_0^t(\partial_s+\mathcal{\bar L}_1+\cL_1+\cL_3^\eps)\psi(s,X_s^{\eps},Y_s^{\eps},\bar X_s,Z_s^{\eps})\dif s\bigg)\bigg|\\
	&\leq\eps\,\mE\big|\big[\psi(0,x,y,x,0)-\psi(0,X_t^{\eps},Y_t^\eps,\bar X_t,Z_t^\eps)\big]\big|\\
	&\quad +\eps\,\mE\bigg|\int_0^t\big[\partial_s\psi(s,X_s^{\eps},Y_s^\eps,\bar X_s,Z_s^{\eps}) -\partial_s\psi(s,X_t^{\eps},Y_t^\eps,\bar X_t,Z_t^\eps)\big]\dif s\bigg|\\
	&\quad+\eps\,\mE\bigg|\int_0^t(\cL_1+\mathcal{\bar L}_1)\psi(s,X_s^{\eps},Y_s^{\eps},\bar X_s,Z_s^{\eps})\dif s\bigg|\\
	&\quad+\eps\,\mE\bigg|\int_0^t\mathcal{L}_3^\eps\psi(s,X_s^{\eps},Y_s^{\eps},\bar X_s,Z_s^{\eps})\dif s\bigg|
	=:\sum_{i=1}^4\sO_i(t,\eps).
	\end{align*}
By making use of Theorem \ref{PP} and Lemma \ref{la41}, we have
\begin{align*}
\sO_1(t,\eps)\leq C_1\,\eps\,\mE\big(1+\|X_t^\eps\|_1+\|Y_t^\eps\|_2^{p}\big)
\leq C_1\,\eps.
\end{align*}
For the second term, since $\phi$ satisfies  (\ref{as1}) and (\ref{as0}), and $t,x,\bar x,z$ all are parameters in equation (\ref{psi1}),  by  Theorem \ref{PP} we get  that $\psi$ satisfies (\ref{as1}) and (\ref{as0}) too,
which together with   Lemmas {\ref{la42}}, {\ref{la43}}, {\ref{bx}}, {\ref{rpz}} and H\"older's inequality implies that for any $\gamma\in(0,1/2)$,
\begin{align*}
\sO_2(t,\eps)
&\leq C_2\,\eps\int_0^t (t-s)^{-1}\Big(\mE(\|\bar X_s-\bar X_t\|_1^3+\|X_s^\eps-X_t^\eps\|_1^3\\
&\qquad\qquad\qquad\qquad\qquad+ \|Y_s^\eps-Y_t^\eps\|_2^3+\|Z_s^\eps-Z_t^\eps\|_1^3)\Big)^{1/3}\dif s\\
&\quad+C_2\,\eps\int_0^t\Big(\mE(\|A\bar X_s\|^2+\|A\bar X_t\|_1^2)\Big)^{1/2}\dif s\\
&\leq C_2\,\eps^{1-\gamma}\leq C_2\,\eps^{1/2}.
\end{align*}
To treat the third term, since for each $t\in [0,T]$, $\psi(t,\cdot,\cdot,\cdot,\cdot)\in C_p^{2,2,2,2}(H_1\times H_1\times H_2\times H_1)$, we have
\begin{align*}
\|(\cL_1+\mathcal{\bar L}_1)\psi(t,x,y,\bar x,z)\|_1\!&\leq\! |\langle Ax+F(x,y),D_x\psi(t,x,y,\bar x,z)\rangle_1|\\
&\quad+\frac{1}{2}\,Tr(Q_1)\|D^2_{ x}\psi(t,x,y,\bar x,z)\|_{\sL(H_1\times H_1,\mR)}\\
&\quad+|\langle A\bar x+\bar F(\bar x),D_{\bar x}\psi(t,x,y,\bar x,z)\rangle_1|\\
&\quad+\frac{1}{2}\,Tr(Q_1)\|D^2_{\bar x}\psi(t,x,y,\bar x,z)\|_{\sL(H_1\times H_1,\mR)}\\
&\leq\! C_3\big(1\!+\!\|A\bar x\|_1\!+\!\|Ax\|_1\!+\!\|y\|_2^p)(1\!+\!\|x\|_1\!+\!\|y\|_2^{p}\big).
\end{align*}
Thus by Lemmas \ref{la41}, \ref{la44} and \ref{bx}, we have
\begin{align*}
\sO_3(t,\eps)
&\leq  C_3\, \eps^{1-\gamma}\leq\!C_3\,\eps^{1/2}.
\end{align*}
For the last term, we write
\begin{align*}
&\sO_{4}(t,\eps)=\eps\,\mE\left(\int_0^t\<AZ_s^\eps,D_z\psi(s,X_s^\eps,Y_s^\eps,\bar X_s,Z_s^\eps)\>_1\dif s\right)\\&+\sqrt{\eps}\,\mE\left(\int_0^t\< \bar F(X_s^\eps)-\bar F(\bar X_s),D_z\psi(s,X_s^\eps,Y_s^\eps,\bar X_s,Z_s^\eps)\>_1\dif s\right)\\
&+\sqrt{\eps}\,\mE\left(\int_0^t\<\delta F(X_s^\eps,Y_s^\eps),D_z\psi(s,X_s^\eps,Y_s^\eps,\bar X_s,Z_s^\eps)\>_1\dif s\right)
=:\sum\limits_{i=1}^3\sO_{4,i}(t,\eps).
\end{align*}
In view of (\ref{as4}), Theorem {\ref{PP}} and (\ref{az}), we have for $\gamma\in(0,1/2\wedge\theta)$,
\begin{align*}
\sO_{4,1}(t,\eps)\leq C_4\,\eps\int_0^t(t-s)^{-1+\gamma}\big(\mE\|(-A)^\gamma Z_s^\eps\|_1^2\big)^{1/2}\dif s\leq C_4\, \eps.
\end{align*}
Furthermore, it is easy to see that
\begin{align*}
\sO_{4,2}(t,\eps)+\sO_{4,3}(t,\eps)\leq C_4\,\sqrt{\eps}\int_0^t\Big(1&+\mE\|X_s^{\eps}\|_1^2\\
&+\mE\|\bar X_s\|_1^2+\mE\|Y_s^{\eps}\|_2^{2p} \Big)\dif s\leq C_4\,\sqrt{\eps}.
\end{align*}
Combining the above computations, we get the desired result.

\vspace{1mm}
\noindent{\bf Step 2.} We proceed to prove estimate (\ref{we2}). By following exactly the same argument as in  the proof of Step 1, we deduce that that for any $\gamma\in(0,1/2)$,
\begin{align*}
&\bigg|\mE\bigg(\frac{1}{\sqrt{\eps}}\int_0^t\!\phi(s,X_s^{\eps},Y_s^{\eps},\bar X_s,Z_s^{\eps})\dif s
-\int_0^t\overline{\delta F\cdot\nabla_z\psi}(s,X_s^{\eps},\bar X_s,Z_s^{\eps})\dif s\bigg)\bigg|\\
&\leq\sqrt{\eps}\,\big|\mE\big[\psi(0,x,y,x,0)-\psi(0,X_t^\eps,Y_t^\eps,\bar X_t,Z_t^\eps)\big]\big|\\
&\quad+\sqrt{\eps}\,\bigg|\mE\left(\int_0^t\left(\partial_t\psi(s,X_s^\eps,Y_s^\eps,\bar X_s,Z_s^{\eps})-\partial_t\psi(s,X_t^\eps,Y_t^\eps,\bar X_t,Z_t^\eps)\right)\dif s\right)\bigg|\\
&\quad+\sqrt{\eps}\,\bigg|\mE\left(\int_0^t(\mathcal{\bar L}_1+\cL_1)\psi(s,X_s^\eps,Y_s^{\eps},\bar X_s,Z_s^{\eps})\dif s\right)\bigg|\\
&\quad+\bigg|\mE\bigg(\sqrt{\eps}\int_0^t\mathcal{L}_3^\eps\psi(s,X_s^\eps,Y_s^{\eps},\bar X_s,Z_s^{\eps})\dif s-\int_0^t\overline{\delta F\cdot\nabla_z\psi}(s,X_s^{\eps},\bar X_s,Z_s^{\eps})\dif s\bigg)\bigg|\\
&\leq C_1\, \eps^{1/2- \gamma}+\bigg|
\mE\bigg(\sqrt{\eps}\int_0^t\mathcal{L}_3^\eps\psi(s,X_s^\eps,Y_s^{\eps},\bar X_s,Z_s^{\eps})\dif s\\&\qquad\qquad\qquad\qquad\qquad-\int_0^t\overline{\delta F\cdot\nabla_z\psi}(s,X_s^{\eps},\bar X_s,Z_s^{\eps})\dif s\bigg)\bigg|.
\end{align*}
Now, for the last term we write
\begin{align*}
&\bigg|\mE\bigg(\sqrt{\eps}\,\int_0^t\mathcal{L}_3^\eps\psi(s,X_s^\eps,Y_s^{\eps},\bar X_s,Z_s^{\eps})\dif s-\int_0^t\overline{\delta F\cdot\nabla_z\psi}(s,X_s^{\eps},\bar X_s,Z_s^{\eps})\dif s\bigg)\bigg|\\
&\leq\sqrt{\eps}\,\bigg|\mE\left(\int_0^t\<AZ_s^\eps,D_z\psi(s,X_s^\eps,Y_s^\eps,\bar X_s,Z_s^\eps)\>_1\dif s\right)\bigg|\\
&\quad+\bigg|\mE\left(\int_0^t\< \bar F(X_s^\eps)-\bar F(\bar X_s),D_z\psi(s,X_s^\eps,Y_s^\eps,\bar X_s,Z_s^\eps)\>_1\dif s\right)\bigg|\\
&\quad+\bigg|\mE\bigg(\int_0^t\Big(\<\delta F(X_s^\eps,Y_s^\eps),D_z\psi(s,X_s^\eps,Y_s^\eps,\bar X_s,Z_s^\eps)\>_1\\&\qquad\qquad\qquad\qquad\qquad-\overline{\delta F\cdot\nabla_z\psi}(s,X_s^\eps,\bar X_s,Z_s^{\eps})\Big)\dif s\bigg)\bigg|
=:\sum\limits_{i=1}^3\tilde\sO_{i}(t,\eps).
\end{align*}
We argue as for $\sO_{4,1}(t,\eps)$ to get that
\begin{align*}
\tilde\sO_{1}(t,\eps)\leq C_2\,\eps^{1/2}.
\end{align*}
Using Theorem \ref{main1}, we further have
\begin{align*}
\tilde\sO_{2}(t,\eps)\leq \!C_3\!\int_0^t\big(1+\mE\| X_s^\eps\|_1^2+\mE\| Y_s^\eps\|_2^{2p}\big)^{1/2}\big(\mE\|X_s^{\eps}-\bar X_s\|_1^2\big)^{1/2}     \dif s\leq \! C_3\,\eps^{1/2}.
\end{align*}
Consequently, we obtain that for any $\gamma\in(0,1/2),$
\begin{align*}
&\bigg|\mE\left(\frac{1}{\sqrt{\eps}}\int_0^t\phi(s,X_s^\eps,Y_s^{\eps},\bar X_s,Z_s^{\eps})\dif s-\int_0^t\overline{\delta F\cdot\nabla_z\psi}(s,X_s^\eps,\bar X_s,Z_s^{\eps})\dif s\right)\bigg|\\
&\leq C_4 \,\eps^{1/2-\gamma}+\bigg|\mE\bigg(\int_0^t\Big(\<\delta F(X_s^\eps,Y_s^\eps),D_z\psi(s,X_s^\eps,Y_s^\eps,\bar X_s,Z_s^\eps)\>_1\\
&\qquad\qquad\qquad\qquad\qquad-\overline{\delta F\cdot\nabla_z\psi}(s,X_s^\eps,\bar X_s,Z_s^{\eps})\Big)\dif s\bigg)\bigg|.
\end{align*}
Note that by the definition of $\overline{\delta F\cdot\nabla_z\psi}$, the function
$$
\tilde \phi(t,x,y,\bar x,z):=\<\delta F(x,y),D_z\psi(t,x,y,\bar x,z)\>_1-\overline{\delta F\cdot\nabla_z\psi}(t,x,\bar x,z)
$$
satisfies the centering condition (\ref{cen222}) and assumption {\bf (H)}. Thus, using (\ref{we1}) directly, we obtain
\begin{align*}
\bigg|\mE\bigg(\int_0^t\Big(\<\delta F(X_s^\eps,Y_s^\eps),&D_z\psi(s,X_s^\eps,Y_s^\eps,\bar X_s,Z_s^\eps)\>_1\\
&-\overline{\delta F\cdot\nabla_z\psi}(s,X_s^\eps,\bar X_s,Z_s^{\eps})\Big)\dif s\bigg)\bigg|\leq C_4\,\eps^{1/2}.
\end{align*}
Combining the above computations, we get the desired result.
\end{proof}

Now, we are in the position to give:

\begin{proof}[Proof of estimate (\ref{nzz})]
	Fix $T>0$ below. Let $\bar u(t,x,z)$ be the solution of the Cauchy problem (\ref{kez}).
	For $t\in[0,T]$, define
	$$
	\tilde u(t,x,z):=\bar u(T-t,x,z).
	$$
	Then it is easy to check that
	$$
	\tilde u(0,x,0)=\bar u(T,x,0)=\mE[\varphi(\bar Z_T)]\quad\text{and}\quad\tilde u(T,x,z)=\bar u(0,x,z)=\varphi(z).
	$$
As a result, by It\^o's formula and (\ref{kez}) we deduce that
	\begin{align*}
	&\mE[\varphi(Z_T^{\eps})]-\mE[\varphi(\bar Z_T)]=\mE[\tilde u(T,\bar X_T,Z_T^{\eps})-\tilde u(0,x,0)]\\&
	=\mE\left(\int_0^T\big(\p_t+\mathcal{\bar L}_1+\cL_3^\eps\big)\tilde u(t,\bar X_t,Z_t^{\eps})\dif t\right)\no\\
	&=\mE\left(\int_0^T(\mathcal{L}_3^\eps-\mathcal{\bar L}_3)\tilde u(t,\bar X_t,Z_t^{\eps})\dif t\right)\no\\
	&=\mE\left(\int_0^T\left\langle \frac{\bar F(X_t^\eps)-\bar F(\bar X_t)}{\sqrt{\eps}}-D_x\bar F( \bar X_t).Z_t^\eps, D_z\tilde u(t,\bar X_t,Z_t^{\eps}))\right\rangle_1 \dif t\right)\\
&\quad+\frac{1}{2}\,\mE\left(\int_0^T  Tr\Big(D_{z}^2\tilde u(t,\bar X_t,Z_t^{\eps})\big[\sigma( X_t^\eps)\sigma^{*}(  X_t^\eps)-\sigma( \bar X_t)\sigma^{*}( \bar X_t)\big]\Big)\dif  t\right)\\
	&\quad+\bigg[\mE\left(\frac{1}{\sqrt{\eps}}\int_0^T\langle \delta F( X_t^\eps,Y_t^\eps), D_z\tilde u(t,\bar X_t,Z_t^{\eps}))\rangle_1 \dif t\right)\\
	&\quad\quad-\frac{1}{2}\,\mE\left(\int_0^T  Tr(D_{z}^2\tilde u(t,\bar X_t,Z_t^{\eps})\sigma(  X_t^\eps)\sigma^{*}(  X_t^\eps))\dif t\right)\bigg]=:\sum\limits_{i=1}^3\sN_i(T,\eps).
	\end{align*}
By the mean value theorem,  H\"older's inequality, (\ref{uz}), (\ref{az}) and Theorem {\ref{main1}}, we deduce that for some $\vartheta\in(0,1)$,
	\begin{align*}
	|\sN_1(T,\eps)|&\leq\mE\bigg(\int_0^T\Big|\big\langle[D_x\bar F(X_t^\eps+\vartheta(X_t^\eps-\bar X_t))\\
&\qquad\qquad\qquad-D_x\bar F( \bar X_t) ]. Z_t^\eps, D_z\tilde u(t,\bar X_t,Z_t^{\eps})\big\rangle_1\Big|\dif t\bigg)\\
	&\leq C_1\int_0^T\big(\mE\|X_t^\eps-\bar X_t\|_1^2\big)^{1/2}\big(\mE\|Z_t^\eps\|_1^2\big)^{1/2}\dif t\leq C_1\,\eps^{1/2}.
	\end{align*}
Furthermore, let $\sU_{t,\bar x,z}(x):=Tr(D_z^2\tilde u(t,\bar x,z)\sigma(x)\sigma^*(x)).$
Then  we have that for every $h\in H_1,$
\begin{align*}
|D_x\sU_{t,\bar x,z}(x).h|\leq C_2(1+\|x\|_1^2)\|h\|_1,
\end{align*}
which together with Theorem {\ref{main1}}, Lemmas \ref{la41} and \ref{bx} yields that
	\begin{align*}
|\sN_2(T,\eps)|\leq C_2\int_0^T\big(1+\mE\|X_t^\eps\|_1^4+\mE\|\bar X_t\|_1^4\big)^{1/2}\big(\mE\|X_t^\eps-\bar X_t\|_1^2\big)^{1/2}\leq C_2\,\eps^{1/2}.
\end{align*}
It remains to control the last term $\sN_3(T,\eps)$. For this purpose, recall that $\Psi$  solves the Poisson equation (\ref{poF}), and define
	$$
	\Phi(t,x,y,\bar x,z):=\<\Psi(x,y),D_z\tilde u(t,\bar x,z)\>_1.
	$$
	Since $\cL_2$ is an operator with respect to the $y$ variable, one can check that $\Phi$ solves the following Poisson equation:
	\begin{align*}
	\mathcal{L}_2(x,y)\Phi(t,x,y,\bar x,z)=-\langle \delta F(x,y),D_z\tilde u(t,\bar x,z)\rangle_1=:-\phi(t,x,y,\bar x,z).
	\end{align*}
It is obvious that $\phi$ satisfies the centering condition (\ref{cen222}). 	Furthermore,
 in view of (\ref{uz2}), (\ref{utz}), (\ref{utzz}) and (\ref{utxz}), we have that for any $t\in[0,T)$, $x,z\in H_1$, $y\in H_2$, $\bar x\in\cD(-A)$ and $h_1,h_2\in H_1$,
 \begin{align*}
	&|\p_t\phi(t,x,y,\bar x,z)|+|D_x\p_t\phi(t,x,y,\bar x,z).h_1|+|D_z\p_t\phi(t,x,y,\bar x,z).h_2|\\
	&\leq|\<\delta F(x,y), \p_tD_{z}\bar u(T-t,\bar x,z)\>_1|+|\<D_x\delta F(x,y).h_1, \p_tD_{z}\bar u(T-t,\bar x,z)\>_1|\\&\quad+ |\p_tD^2_{z}\bar u(T-t,\bar x,z).(\delta F(x,y),h_2)|\\
&\leq\! C_3\,(T-t)^{-1}(\|h_1\|_1+\|h_2\|_1)(1+\|A\bar x\|_1+\|\bar x\|_1^2+\|z\|_1)
(1\!+\!\|x\|_1\!+\!\|y\|_2^p),
	\end{align*}
and for any $h_3\in\cD(-A),$
\begin{align*}
&|D_{\bar x}\p_t\phi(t,x,y,\bar x,z).h_3|=\p_tD_{\bar x}D_z\bar u(T-t,\bar x,z).(\delta F(x,y),h_3)\\
&\leq\!C_3\,\Big((T-t)^{-1}(1+\|A\bar x\|_1+\|\bar x\|_1^2+
\|z\|_1)\|h_3\|_1+\|Ah_3\|_1\Big)\\
&\qquad\quad\times(1+\|x\|_1+\|y\|_2^p),
\end{align*}
and for any $h\in \cD((-A)^{\vartheta})$ with $\vartheta\in[0,1],$
	\begin{align*}
	|D_z\phi(t,x,y,\bar x,z).(-A)^\vartheta h|&= D^2_{z}\bar u(T-t,\bar x,z).(\delta F(x,y),(-A)^\vartheta h)\no\\&\leq  C_3\,(T-t)^{-\vartheta}(1+\|x\|_1+\|y\|_2^{p})\|h\|_1.
	\end{align*}
	Furthermore, by the definition of $\sigma$ in (\ref{sst}),
	we have
	\begin{align*}
	\overline{\delta F\cdot\nabla_z\Phi}(t,x,\bar x,z)&=\int_{H_2} D_z\Phi(t,x,y,\bar x,z).\delta F(x,y)\mu^x(\dif y)\\
&=\int_{H_2}D_z^2\tilde u(t,\bar x,z).(\Psi(x,y),\delta F(x,y))\mu^x(\dif y)\\
&=\frac{1}{2}Tr(D_z^2\tilde u(t,\bar x,z)\sigma(x)\sigma^*(x)).
	\end{align*}
	Thus, it follows by (\ref{we2}) directly that for any $\gamma\in(0,1/2),$
	$$
	|\sN_3(T,\eps)|\leq C_3\; \eps^{1/2- \gamma}.
	$$
Combining the above computations, we get the desired result.
\end{proof}

\section{Appendix}

Throughout this section, we assume that ({\bf A1}) and ({\bf A2})   hold, $F\in C^{1,\eta}_p(H_1\times H_2,H_1)$ and $G\in C^{1,\eta}_b(H_1\times H_2,H_2)$ with $\eta>0$. We have the following result.

\begin{lemma}\label{la41}
For any $(x,y)\in H_1\times H_2$, there exists  a unique mild solution for the equation (\ref{spde1}), i.e., for every  $t\geq0$,
	\begin{equation} \label{spde22}
	\left\{ \begin{aligned}
	&X^{\eps}_t =e^{tA}x+\int_0^te^{(t-s)A}F(X^{\eps}_s, Y^{\eps}_s)\dif s+\int_0^te^{(t-s)A}\dif W^1_s,\\
	& Y^{\eps}_t =e^{\frac{t}{\eps}B}y+\eps^{-1}\int_0^te^{\frac{t-s}{\eps}B}G(X^{\eps}_s, Y^{\eps}_s)\dif s+\eps^{-1/2}\int_0^te^{\frac{t-s}{\eps}B}\dif W_s^2.\\
	\end{aligned} \right.
	\end{equation}
Moreover, for any $T>0$, $q\geq1$ and $x\in \cD((-A)^{\theta})$ with $\theta\in[0,1)$, we have
\begin{align}\label{msx}
	\sup\limits_{\eps\in(0,1)}\sup\limits_{t\in[0,T]}\mE\| (-A)^\theta X^{\eps}_t\|_1^{q}\leq C_{\theta,q,T}\,\big(1+\| x\|_{(-A)^\theta}^{q}+\| y\|_2^{pq}\big)
	\end{align}
	and
\begin{align}\label{msy}
	\sup\limits_{\eps\in(0,1)}\sup\limits_{t\in[0,T]}\mE\| Y^{\eps}_t\|_2^{q}\leq C_{q,T}\,(1+\| y\|_2^{q}),
	\end{align}
where $C_{\theta,q,T},C_{q,T}>0$ are constants.
\end{lemma}

\begin{proof}
The well-posedness  of SPDE (\ref{spde1}) with H\"older continuous coefficients follows from \cite[Theorem 7]{DF}. Furthermore, estimate (\ref{msx}) can be proved similarly as in \cite[Proposition 2.10]{Br2} or \cite[Proposition 4.3]{Ce}, and estimate (\ref{msy}) can be proved as  in \cite[Proposition 4.2]{Ce}. We omit the details here.
\end{proof}

Note that estimate (\ref{msx}) holds only for $\theta<1$. In order to get the estimate for $\theta=1$, we need some extra regularity results for $X_t^\eps$ and $Y_t^\eps$ with respect to  the time variable. The following two results extend \cite[Proposition 4.4]{Ce} and \cite[Proposition A.4]{Br1}, respectively.

\bl\label{la42}
Let $T>0$, $\gamma\in[0,1]$, $x\in \cD((-A)^\theta)$ with $\theta\in[0,\gamma]$ and $y\in H_2$. Then for  every  $q\geq1$ and $0< s\leq t\leq T,$	 we have
\begin{align*}
	\Big(\mE\|  X_t^\eps- X_s^\eps\|_1^q\Big)^{\frac{1}{q}}\leq C_{\theta,\gamma,q,T}\,&\bigg(\frac{{(t-s)}^{\gamma}}{s^{\gamma-\theta}}e^{-\frac{\alpha_1}{2}s}\|x\|_{(-A)^\theta} \\
&\quad+(t-s)^{\frac{1}{2}}\big(1+\|x\|_1+\|y\|_2^p\big)\bigg),
	\end{align*}
where $C_{\theta,\gamma,q,T}>0$ is a constant.
\el

\begin{proof}
In view of (\ref{spde22}), we have
\begin{align}\label{xi5}
X_t^{\eps}-X_s^{\eps}&=(e^{tA}-e^{sA})x+\int_s^te^{(t-r)A}F(X_r^{\eps},Y_r^{\eps})\dif r\no\\
&\quad+\int_0^s(e^{(t-r)A}-e^{(s-r)A})F(X_r^{\eps},Y_r^{\eps}))\dif r+\int_s^te^{(t-r)A}\dif W_r^1\no\\
&\quad+\int_0^s(e^{(t-r)A}-e^{(s-r)A})\dif W_r^1=:\sum\limits_{i=1}^5\sX_i(t,s).
\end{align}
Below, we estimate each term on the right hand side of (\ref{xi5}) separately. For the first term, by Proposition \ref{arp} (iii)  we easily get
$$\|\sX_1(t,s)\|_1\leq C_1\,\frac{{(t-s)}^{\gamma}}{s^{\gamma-\theta}} e^{-\frac{\alpha_1}{2}s}\|x\|_{(-A)^\theta}.$$
For the second term, by Minkowski's inequality and Lemma \ref{la41}, we deduce that
\begin{align*}
\mE\|\sX_2(t,s)\|_1^q
&\leq\left(\int_s^t\Big(\mE\|e^{(t-r)A}F(X_r^{\eps},Y_r^{\eps})\|_1^q\Big)^{1/q}\dif r\right)^{q}\\
&\leq C_2\,(t-s)^q\big(1+\|x\|_1^q+\|y\|_2^{pq}\big).
\end{align*}
Similarly, using Proposition \ref{arp} (ii), Lemma \ref{la41} and Minkowski's inequality again, we have
\begin{align*}
\mE\|\sX_3(t,s)\|_1^q
&\leq\left(\int_0^s\Big(\mE\|(e^{(t-s)A)}-I)e^{(s-r)A}F(X_r^{\eps},Y_r^{\eps})\|_1^q
\Big)^{1/q}\dif r\right)^{q}\\
&\leq C_3\,(t-s)^{q/2}\left(\int_0^s\Big(\mE\|(-A)^{1/2} e^{(s-r)A}F(X_r^{\eps},Y_r^{\eps})\|_1^q\Big)^{1/q}\dif r\right)^{q}\\
&\leq C_3\,(t-s)^{q/2}\big(1+\|x\|_1^q+\|y\|_2^{pq}\big).
\end{align*}
Using Burkholder-Davis-Gundy's inequality and the assumption {\bf (A2)}, we further get
\begin{align*}
\mE\|\sX_4(t,s)\|_1^q\leq C_4\,\left(\int_s^t\|e^{(t-r)A}Q_1^{1/2}\|_{\sL_2(H_1)}^2\dif r\right)^{q/2}\leq C_4\,(t-s)^{q/2},
\end{align*}
and
\begin{align*}
\mE\|\sX_5(t,s)\|_1^q
&\leq C_5\,(t-s)^{q/2}\left(\int_0^s\|(-A)^{1/2} e^{(s-r)A}Q_1^{1/2}\|_{\sL_2(H_1)}^2\dif r\right)^{q/2}\\
&\leq C_5\, (t-s)^{q/2}.
\end{align*}
Combining the above computations, we get the desired result.
\end{proof}

\bl\label{la43}
Let $T>0$, $\gamma\in[0,1/2]$, $x\in H_1$ and $y\in \cD((-B)^\theta)$ with $\theta\in[0,\gamma]$. Then for  every  $q\geq1$ and $0< s\leq t\leq T,$	 we have
\begin{align*}
\big(\mE\|  Y_t^\eps- Y_s^\eps\|_2^{q}\big)^{\frac{1}{q}}\leq C_{\theta,\gamma,q,T}\,\bigg(\frac{(t-s)^{\gamma}}{s^{\gamma-\theta}\eps^{\theta}}e^{-\frac{\mu_1}{2\eps}s}\|y\|_{(-B)^\theta}
+\frac{(t-s)^{\gamma}}{\eps^{\gamma}}\bigg),
	\end{align*}
where $C_{\theta,\gamma,q,T}>0$ is a constant.
\el

\begin{proof}
In view of  (\ref{spde22}), we have
\begin{align*}
Y_t^{\eps}-Y_s^{\eps}&=(e^{\frac{t}{\eps}B}-e^{\frac{s}{\eps}B})y+\frac{1}{\eps}\int_s^te^{\frac{(t-r)}{\eps}B}G(X_r^{\eps},Y_r^{\eps})\dif r\\
&\quad+\frac{1}{\eps}\int_0^s(e^{\frac{(t-r)}{\eps}B}-e^{\frac{(s-r)}{\eps}B})G(X_r^{\eps},Y_r^{\eps}))\dif r+\frac{1}{\sqrt{\eps}}\int_s^te^{\frac{(t-r)}{\eps}B}\dif W_r^2\\
&\quad+\frac{1}{\sqrt{\eps}}\int_0^s(e^{\frac{(t-r)}{\eps}B}-e^{\frac{(s-r)}{\eps}B})\dif W_r^2=:\sum\limits_{i=1}^5\sY_i(t,s)
\end{align*}
In exactly the same way as in the proof of Lemma \ref{la42}, we deduce that
$$
\|\sY_1(t,s)\|_2\leq C_1\,\frac{{(t-s)}^{\gamma}}{s^{\gamma-\theta}\eps^{\theta}} e^{-\frac{\mu_1}{2\eps}s}\|y\|_{(-B)^\theta},
$$
and for any $\gamma\in[0,1/2]$,
\begin{align*}
\mE\|\sY_2(t,s)\|_2^q\leq C_2\!\left(\frac{1}{\eps}\int_s^te^{\frac{-\mu_1}{2\eps}(t-r)}\Big
(\mE\|G(X_r^{\eps},Y_r^{\eps})\|_2^q\Big)^{1/q}\dif r\right)^{q}\!\leq C_2\,\frac{(t-s)^{\gamma q}}{\eps^{\gamma q}},
\end{align*}
and
\begin{align*}
\mE\|\sY_3(t,s)\|_1^q
&\leq\left(\frac{1}{\eps}\int_0^s\Big(\mE\|(e^{\frac{(t-s)}{\eps}B)}-I)e^{\frac{(s-r)}{\eps}B} G(X_r^{\eps},Y_r^{\eps})\|_2^q\Big)^{1/q}\dif r\right)^{q}\\
&\leq C_3\,\frac{(t-s)^{\gamma q}}{\eps^{\gamma q}}\left(\int_0^{s/\eps}\|(-B)^\gamma e^{rB}\|_{\sL(H_2)}\dif r\right)^{q}\\
&\leq C_3\,\frac{(t-s)^{\gamma q}}{\eps^{\gamma q}}\left(\int_0^{s/\eps}r^{-\gamma}e^{\frac{-\mu_1}{2}r}\dif r\right)^{q}\leq C_3\,\frac{(t-s)^{\gamma q}}{\eps^{\gamma q}}.
\end{align*}
To control the last two terms, by Burkholder-Davis-Gundy's inequality and the assumption ({\bf A2}), we deduce that  for any $\gamma\in[0,1],$
\begin{align*}
\mE\|\sY_4(t,s)\|_1^q&\leq C_4\,\left(\frac{1}{\eps}\int_s^t\|e^{\frac{(t-r)}{\eps}B}Q_2^{1/2}\|_{\sL_2(H_2)}^2\dif r\right)^{\frac{q}{2}}\\
&\leq C_4\,\left(\sum\limits_{n=1}^\infty\lambda_{2,n}\mu_n^{\gamma-1} \frac{(t-s)^\gamma}{\eps^\gamma}\right)^{q/2}\leq C_4\,\frac{(t-s)^\frac{\gamma q}{2}}{\eps^\frac{\gamma q}{2}},
\end{align*}
and for $\gamma\in[0,1/2],$
\begin{align*}
\mE\|\sY_5(t,s)\|_1^q
&\leq C_5\,\frac{(t-s)^{\gamma q}}{\eps^{\gamma q}}\left(\frac{1}{\eps}\int_0^s\|(-B)^\gamma e^{\frac{(s-r)}{\eps}B}Q_2^{1/2}\|^2_{\sL_2(H_2)}\dif r\right)^{q/2}\\
&\leq C_5\,\frac{(t-s)^{\gamma q}}{\eps^{\gamma q}}\left(\sum\limits_{n=1}^\infty\lambda_{2,n}\mu_n^{2\gamma-1}\right)^{q/2}\leq C_5\,\frac{(t-s)^{\gamma q}}{\eps^{\gamma q}}.
\end{align*}
Combining the above computations, we get the desired result.
\end{proof}

Now we can prove the following moment estimate.

\bl\label{la44}
Let $T>0$, $x\in \cD((-A)^\theta)$ with $\theta\in[0,1]$ and $y\in H_2$. Then for any $q\geq1$, $\gamma>0$ and $0\leq t\leq T$, we have
	\begin{align*}
\big(\mE\| A  X_t^\eps\|_1^q\big)^{1/q}\leq C_{\theta,\gamma,q,T}\Big(t^{(\theta-1)}+\eps^{-\gamma }\Big)\big(1+\|x\|^2_{(-A)^\theta}+\|y\|_{2}^{2p}\big),
	\end{align*}
where $C_{\theta,\gamma,q,T}>0$ is a constant.
\el

\begin{proof}
We have
\begin{align*}
AX_t^{\eps}&=Ae^{tA}x+\int_0^tAe^{(t-s)A}F(X_t^{\eps},Y_t^{\eps})\dif s\\
&\quad+\int_0^tAe^{(t-s)A}\big(F(X_s^{\eps},Y_s^{\eps})-F(X_t^{\eps},Y_t^{\eps})\big)\dif s\\
&\quad+\int_0^tAe^{(t-s)A}\dif W_s^1=:\sum\limits_{i=1}^4\sX_i(t,\eps)
\end{align*}
By Proposition \ref{arp} (i), we easily see that
$$
\|\sX_1(t,\eps)\|_1\leq C_1\,t^{(\theta-1)}\|x\|_{(-A)^\theta}.
$$
For the second term, note that
\begin{align*}
\int_0^tAe^{(t-s)A}F(X_t^{\eps},Y_t^{\eps})\dif s&=\int_0^t\p_te^{(t-s)A}F(X_t^{\eps},Y_t^{\eps})\dif s\\
&=-(e^{tA}-I)F(X_t^{\eps},Y_t^{\eps}),
\end{align*}
hence we deduce that
\begin{align*}
\mE\|\sX_2(t,\eps)\|_1^q\leq C_2\,(1+\mE\|X_t^{\eps}\|_1^q+\mE\|Y_t^{\eps}\|_2^{pq})\leq C_2\,\big(1+\|x\|_1^q+\|y\|_2^{pq}\big).
\end{align*}
Furthermore, by applying Lemmas \ref{la42} and \ref{la43} with $\theta=0$, we get that for any $\gamma\in(0,1/2]$,
\begin{align*}
\mE\|\sX_3(t,\eps)\|_1^q&\leq \! C_3\,(1+\|x\|_1^q+\|y\|_2^{pq})\Bigg(\int_0^t(t-s)^{-1}\bigg[\Big(\mE\big[\|X^\eps_t- X^\eps_s\|_1^{2q}\big]\Big)^{1/2q}\\
&\qquad\qquad\qquad\qquad\qquad\qquad+\Big(\mE\big[\|Y_t^{\eps}-Y_s^{\eps}\|_2^{2\eta q}\big]\Big)
^{1/2q}\bigg]\dif s\Bigg)^q\\
&\leq  C_3\,\big(1+\|x\|_1^{2q}+\|y\|_2^{2pq}\big)\bigg(\int_0^t(t-s)^{\eta\gamma-1}\Big(\frac{1}{s^{\eta\gamma}}+\frac{1}{\eps^{\eta\gamma}}\Big)\dif s\bigg)^q\\
&\leq C_3\,\eps^{-\gamma q}\big(1+\|x\|_1^{2q}+\|y\|_2^{2pq}\big).
\end{align*}
Finally, by Burkholder-Davis-Gundy's inequality and assumption {\bf(A2)}, we have
\begin{align*}
\mE\|\sX_4(t,\eps)\|_1^q\leq  C_4\left(\int_0^t\|Ae^{(t-s)A}Q_1^{1/2}\|^2_{\sL_2(H_1)}\dif s\right)^{q/2}\leq C_4.
\end{align*}
The conclusion follows by the above estimates.
\end{proof}

The following results for the averaged equation can be proved as in Lemmas \ref{la41}, {\ref{la42}} and {\ref{la44}}, so we omit the details here.

\begin{lemma}\label{bx}
	 For  $x\in H_1$, the averaged equation (\ref{spde2}) has a unique mild solution, i.e., for all $t>0$,
	\begin{align}\label{msb}\bar{X}_t =e^{tA}x+\int_0^te^{(t-s)A}\bar{F}(\bar{X}_s)\dif s+\int_0^te^{(t-s)A}\dif W^1_s.\end{align}
In addition,  we have:

\vspace{1mm}
\noindent(i) For any  $q\geq1$ and $x\in \cD(-A)^{\theta}$ with $\theta\in[0,1)$,
	\begin{align*}
	\sup\limits_{t\in[0,T]}\mE\| (-A)^{\theta}\bar X_t\|_1^q\leq C_{\theta,q,T}(1+\| x\|_{(-A)^\theta}^q);
	\end{align*}
(ii) For any $q\geq1$, $\theta\in[0,1]$ and $0\leq t\leq T,$
\begin{align*}
	(\mE\| A\bar X_t\|_1^q)^{1/q}\leq C_{\theta,q,T}(1+t^{\theta-1}\| x\|_{(-A)^\theta});
	\end{align*}
(iii) For any $q\geq1$, $\gamma\in[0,1]$, $x\in \cD((-A)^\theta)$ with $\theta\in[0,\gamma]$  and $0< s\leq t\leq T,$	
\begin{align*}
	\Big(\mE\|  \bar X_t- \bar X_s\|_1^q\Big)^{\frac{1}{q}}&\leq C_{\theta,\gamma,q,T}\,\bigg(\frac{{(t-s)}^{\gamma}}{s^{\gamma-\theta}}e^{-\frac{\alpha_1}{2}s}\|x\|_{(-A)^\theta} +(t-s)^{\frac{1}{2}}\big(1+\|x\|_1\big)\bigg),
	\end{align*}
where $C_{\theta,q,T}, C_{\theta,\gamma,q,T}>0$ are constants.
\end{lemma}

\noindent{\bf {Acknowledgements:}} This work is supported  by the DFG through CRC 1283,  the Alexander-von-Humboldt foundation  and NSFC (No. 11701233, 11931004).


\end{document}